\documentclass[leqno,11pt]{amsart}

\usepackage{geometry}

\usepackage[all,cmtip]{xy}
\usepackage{tikz-cd}

\usepackage[all]{xy} 

\usepackage{comment}

\usepackage{mathtools}
\usepackage{amsmath, amssymb, amsfonts, latexsym, mdwlist, amsthm}
\usepackage{subfig}
\usepackage{graphicx}
\usepackage{wrapfig}

\usepackage[bookmarks, colorlinks, breaklinks, pdftitle={},
pdfauthor={}]{hyperref}
\hypersetup{linkcolor=blue,citecolor=blue,filecolor=black,urlcolor=blue}

\usepackage{tikz}
\usetikzlibrary{calc,trees,positioning,arrows,chains,shapes.geometric,%
    decorations.pathreplacing,decorations.pathmorphing,shapes,%
    matrix,shapes.symbols}

\tikzset{
>=stealth',
  punktchain/.style={
    rectangle,
    rounded corners,
    draw=black, thick,
    minimum height=3em,
    text centered,
    on chain},
  line/.style={draw, thick, <-},
  element/.style={
    tape,
    top color=white,
    bottom color=blue!50!black!60!,
    minimum width=8em,
    draw=blue!40!black!90, very thick,
    text width=10em,
    minimum height=3.5em,
    text centered,
    on chain},
  every join/.style={->, thick,shorten >=1pt},
  decoration={brace},
  tuborg/.style={decorate},
  tubnode/.style={midway, right=2pt},
}

\usepackage{paralist}
\setdefaultenum{(a)}{(i)}{}{}
\usepackage{enumitem} 
\usepackage{graphicx}

\usetikzlibrary{patterns}

\renewcommand\_{^{}_}
\renewcommand\;{\hspace{.6pt}}
\newcommand\To{\longrightarrow}
\newcommand\into{\hookrightarrow}
\newcommand{\Into}{\ensuremath{\lhook\joinrel\relbar\joinrel\rightarrow}}
\newcommand\Onto{\longrightarrow\hspace{-5.5mm}\longrightarrow}
\newcommand\Mapsto{\ensuremath{\shortmid\joinrel\relbar\joinrel\rightarrow}}
\newcommand{\rt}[1]{\xrightarrow{\ #1\ }}
\newcommand\PP{\mathbb P}
\newcommand\C{\mathbb C}
\newcommand\Q{\mathbb Q}
\newcommand\R{\mathbb R}
\newcommand\N{\mathbb N}
\newcommand\Z{\mathbb Z}
\newcommand\J{\mathsf J\;}

\newcommand\cA{\mathcal A}
\newcommand\cD{\mathcal D}
\newcommand\cH{\mathcal H}
\newcommand\cO{\mathcal O}

\newcommand\cM{\mathcal M}

\newcommand\ch{\operatorname{ch}}

\newcommand\Hom{\operatorname{Hom}}
\newcommand\Pic{\operatorname{Pic}}
\newcommand\Ext{\operatorname{Ext}}
\newcommand\rk{\operatorname{rank}}
\newcommand\Coh{\operatorname{Coh}}
\newcommand\cok{\operatorname{coker}}
\newcommand\js{\operatorname{JS}}
\renewcommand\v{\mathsf v}
\newcommand\Ab{\mathcal A_{\;b}}
\newcommand\nubw{\nu\_{b,w}}
\newcommand\vno{v_{n_{\raisebox{-1pt}{\scalebox{0.6}{$0$}}}}}
\newcommand\im{\operatorname{im}}
\newcommand\ext{\mathcal E\hspace{-1pt}xt}
\renewcommand\hom{\mathcal H\hspace{-1pt}om}
\newcommand\udot{^\bullet}

\newcommand\onto{\to\hspace{-3mm}\to}
\newcommand\so{\operatorname{\Longrightarrow}}
\renewcommand\({\big(}
\renewcommand\){\big)}
\renewcommand\={\ =\ }
\newcommand\arXiv[1]{\href{http://arxiv.org/abs/#1}{arXiv:#1}}
\newcommand\mathAG[1]{\href{http://arxiv.org/abs/math/#1}{math.AG/#1}}

\newcommand\w{\operatorname{w}}

\newcommand\cl{\operatorname{cl}\;}

\newcommand\beq[1]{\begin{equation}\label{#1}}
\newcommand\eeq{\end{equation}}
\newcommand\beqa{\begin{eqnarray*}}
\newcommand\eeqa{\end{eqnarray*}}

\makeatletter
\newtheorem*{rep@theorem}{\rep@title}
\newcommand{\newreptheorem}[2]{%
\newenvironment{rep#1}[1]{%
 \def\rep@title{#2 \ref{##1}}%
 \begin{rep@theorem}}%
 {\end{rep@theorem}}}
\makeatother

\newtheorem{Thm}{Theorem}[section]
\newreptheorem{Thm}{Theorem}
\newtheorem{Thm*}{Theorem}
\newtheorem{Prop}[Thm]{Proposition}

\newtheorem{Lem}[Thm]{Lemma}

\newreptheorem{Cor}{Corollary}
\newtheorem{Con}[Thm]{Conjecture}

\newtheorem{thm-int}{Theorem}

\theoremstyle{definition}
\newtheorem{Def-s}[Thm]{Definition}
\newtheorem{Def}[Thm]{Definition}
\newtheorem{Rem}[Thm]{Remark}

\newcommand{\ignore}[1]{}

\begin{document}

\title{Rank $r$ DT theory from rank $1$}
\author{S. Feyzbakhsh and R. P. Thomas}
\maketitle
\begin{abstract}
Fix a Calabi-Yau 3-fold $X$ satisfying the Bogomolov-Gieseker conjecture of Bayer-Macr\`i-Toda, such as the quintic 3-fold.

We express Joyce’s generalised DT invariants counting Gieseker semistable sheaves of any rank $r$ on $X$ in terms of those counting sheaves of rank 1. By the MNOP conjecture they are therefore determined by the Gromov-Witten invariants of $X$.
\end{abstract}
\bigskip

Let $X$ be a smooth projective Calabi-Yau 3-fold satisfying the Bogomolov-Gieseker conjecture of Bayer-Macr\`i-Toda \cite{BMT}. We show that the higher rank or ``nonabelian" DT theory of $X$ is governed by its rank one ``abelian" theory. (``Nonabelian" and ``abelian" refer to the gauge groups $U(r)$ and $U(1)$ respectively.) This can be thought of as a 6-dimensional analogue of the correspondence between nonabelian Donaldson theory and abelian Seiberg-Witten theory for smooth 4-manifolds.

Combined with the MNOP conjecture \cite{MNOP}, now proved for most Calabi-Yau 3-folds \cite{PP}, this expresses any DT invariant $\J(v)$ entirely in terms of the Gromov-Witten invariants of $X$. Here $\J(v)\in\Q$ denotes Joyce-Song's generalised DT invariant \cite{JS} counting Gieseker semistable sheaves on $X$ of numerical K-theory class $v$ of rank $r>0$.

\begin{Thm*}\label{1} Let $(X,\cO_X(1))$ be a Calabi-Yau 3-fold satisfying the conjectural Bogomolov-Gieseker inequality of \cite{BMT}. Then for fixed $v$ of any rank $r>0$,
\beq{formu}
\J(v)\=F\(\J(\alpha_1),\J(\alpha_2),\dots\)
\eeq
is a universal polynomial in invariants $\J(\alpha_i)$, with all $\alpha_i$ of rank 1. If $X$ also satisfies the MNOP conjecture then we can replace the $\J(\alpha_i)$ by the Gromov-Witten invariants of $X$.\footnote{In particular if $X_1,\,X_2$ are symplectomorphic and both satisfy the BG and MNOP conjectures then their DT theories coincide. It has long been expected that DT invariants should be symplectic invariants.}
\end{Thm*}

The coefficients of $F$ depend only on $H^*(X,\Q)$ as a graded ring with pairing, $\ch(v)$, the Chern classes of $X$, and the class $H:=c_1(\cO_X(1))$ used to define Gieseker stability. There are countably many terms in the formula \eqref{formu} but only finitely many are nonzero.

\begin{Thm*}\label{2}
If $H^1(\cO_X)=0$ then Theorem \ref{1} also holds for classes $v$ of rank $r=0$.
\end{Thm*}

\subsection*{Joyce-Song pairs}
We fix any smooth complex projective threefold $(X,\cO_X(1))$ (not necssarily Calabi-Yau). A \emph{Joyce-Song stable pair} $(F,s)$ of class $v$ on $X$ consists of
\begin{itemize}
\item a rank $r$ semistable sheaf $F$ of class $v$, a fixed $n\gg0$, and
\item $s\in H^0(F(n))$ which factors through no semi-destabilising subsheaf of $F$.
\end{itemize}
For us ``semistable" will refer to a specific weak stability condition of \cite{BMT, BMS} (in contrast to the Gieseker stability used by Joyce-Song). When $r\ge1$ this is enough to ensure that $s$ is \emph{injective}, so it makes sense to consider its cokernel $E$,
\beq{JSs}
0\To\cO_X(-n)\rt{s}F\To E\To0,
\eeq
of class $v_n:=v-[\cO_X(-n)]$. This allows us to relate the counting of the semistable sheaves $F$ of rank $r$ to the counting of sheaves $E$ of rank $r-1$. We can then use wall crossing to try to move from weakly semistable $E,F$ to Gieseker semistable $E,F$. In this way we expressed rank $r\ge1$ DT invariants $\J(v)$ in terms of invariants $\J$ of ranks $\le r-1$ for
\begin{itemize}
\item ideal sheaves of curves (so $r=1$) in \cite{FT1},
\item arbitrary sheaves $F$ of rank $r=1$ in \cite{FT2}, and
\item arbitrary sheaves $F$ of arbitrary rank $r>1$ in \cite{FT3}.
\end{itemize}
This does not give Theorem 1, however. Instead, by inducting on $r=\rk(v)$, it expresses $\J(v)$ in terms of invariants of \emph{ranks 1 and 0}. Or, applying it one more time to the rank 1 invariants, it expresses $\J(v)$ in terms of counts of Gieseker semistable sheaves of \emph{rank 0 and pure dimension 2} \cite[Theorem 1]{FT3}. Thus to prove Theorem \ref{1} it therefore remains\footnote{We may assume $H^1(\cO_X)=0$ in Theorem \ref{1}; otherwise the Jac$\;(X)$ action on moduli spaces of sheaves forces $\J(v)=0$ for all classes $v$ of $\rk>0$. We focus on Theorem \ref{2} for classes $v$ of $\dim=2$ . For the $\dim v=1$ case see Section \ref{dim1}. For $\dim v=0$ \cite[Equation 6.19]{JS} expresses $\J(v)$ in terms of $e(X)=c_3(X)$.} to prove Theorem \ref{2} for \emph{rank 0 dimension 2} classes $v$, i.e. those with $\ch_1(v).H^2>0$.

In this paper we go one step further, from rank 0 to $-1$. We replace rank 0 sheaves $F$ by the cones $E\in\cD(X)$ on their Joyce-Song pairs. These are nontrivial complexes as $s$ has both kernel and cokernel. The exact sequence \eqref{JSs} is replaced by the exact triangle
\beq{trangl}
\cO_X(-n)\rt{s}F\To E\To \cO_X(-n)[1].
\eeq
By further wall crossing, dualising and shifting by $[1]$ we can then relate the semistable complexes $E$ to rank 1 ideal sheaves to prove Theorem 2 for rank 0 dimension 2 classes. 

\subsection*{Wall crossing} Most of the work is in Section \ref{shvs}, which applies to \emph{any} smooth projective 3-fold satisfying the Bogomolov-Gieseker conjecture. We use the weak stability conditions of \cite{BMT, BMS}, finding the walls of instability for objects of class
$$
v_n\=[E]\=[F]-[\cO_X(-n)]\=v-[\cO_X(-n)].
$$
The most important we call its \emph{Joyce-Song wall} $\ell_{\js}$ on which the slopes of $E,\,F$ and $\cO_X(-n)[1]$ coincide.
Below $\ell_{\js}$ the exact triangle \eqref{trangl} destabilises $E$, while above it such complexes become semistable. The same holds when we replace $\cO_X(-n)[1]$ by $T(-n)[1]$ for any line bundle $T$ with torsion first Chern class,
	\beq{Pic0}
	T\ \in\ \Pic\_0(X)\ :=\ \big\{L\in\Pic(X)\ \colon\ c_1(L)\,=\,0\,\in\,H^2(X,\Q)\big\}.
	\eeq

We show that all other walls of instability for the class $[E]$ are similar: whenever $E$ is destabilised below the wall it is by a triangle $E_0\to E\to E_1$ with $E_0$ a rank 0 sheaf and $E_1$ a rank $-1$ complex (with cohomology in degrees $-1$ and $0$ only), while above the wall it is extensions in the opposite direction that are unstable.

\subsection*{Wall crossing formula} From Section \ref{wcross} we assume $K_X\cong\cO_X$ and $H^1(\cO_X)=0$. Thus Joyce-Song's generalised DT invariants are defined and satisfy a wall crossing formula.

The wall crossing formula calculates the change in the invariants counting semistable objects in class $[E]$, as we cross any of the walls described above, in terms of the counting invariants of the classes $[E_0]$ and $[E_1]$ (and of the classes of semistable factors of $E_0,\,E_1$ of the same slope). For instance suppose for simplicity that, on crossing the Joyce-Song wall $\ell_{\js}$, linear combinations of $v$ and $\big[\cO_X(-n)[1]\big]$ are the only classes whose slope crosses that of $v_n$, and that $v$ is primitive. Then the wall crossing formula is
\beq{1stwcf}
\J_{b,w_+}(v_n)\=\J_{b,w_-}(v_n) +(-1)^{\chi(v(n))-1}\chi(v(n))\cdot\#H^2(X,\Z)_{\mathrm{tors}}\cdot\J_{b,w_+}(v),
\eeq
where $(b,w_{\pm})$ are points just above and below the wall $\ell_{\js}$ in the space of weak stability conditions, the $\J_{b,w}$ are the corresponding invariants counting semistable objects, and $\#H^2(X,\Z)_{\mathrm{tors}}$ is the invariant counting the objects $T(-n)[1]$ for $T\in\Pic\_0(X)$. Finally $(-1)^{\chi(v(n))-1}\chi(v(n))$ counts the extensions \eqref{trangl} by taking the signed topological Euler characteristic of the space $\PP\(H^0(F(n))\)$ of all of them.

\subsection*{From small to large volume} The Bogomolov-Gieseker inequality gives a region in the space of weak stability conditions in which there are no semistable objects $E$. Conversely, by work of Toda \cite{TodaBG}, there is a region (``the large volume chamber") in which the rank $-1$ objects $E$ (or $E_1$) are semistable if and only if the rank 1 complex $E^\vee\otimes(\det E)^{-1}[1]$ is a \emph{stable pair} in the sense of \cite{PT}. Using the wall crossing formula over the walls described above we move from the first of these regions (where the invariants vanish) to the second (where they equal stable pair invariants).

Just as \eqref{1stwcf} relates $\J_{b,w}(v)$ to the rank $-1$ invariants $\J_{b,w}(v_n)$, the eventual result is an analogous --- but more complicated --- formula relating invariants counting rank 0 sheaves such as $F$ and $E_0$ to stable pair invariants counting rank $-1$ complexes such as $E$ and $E_1$. The crucial fact, proved in Propositions \ref{c' negative} and \ref{T}, is that the degree $\ch_1\!.\;H^2$ of the support of these rank 0 sheaves is $\le\ch_1(v).H^2$ with equality only on $\ell_{\js}$. So working by induction on $\ch_1\!.\;H^2$ we may assume the lower order terms --- counting rank 0 sheaves $E_0$ with $\ch_1\!.\;H^2<\ch_1(v).H^2$ --- have already been expressed in terms of stable pair invariants. We thus conclude that so too is the invariant counting the sheaves $F$ \eqref{trangl} in class $v$.

By a further wall crossing already carried out in \cite{BrDTPT, TodaDTPT}, stable pair invariants can be expressed in terms of the rank 1 DT invariants counting ideal sheaves. And a final wall crossing from \cite[Section 5]{FT3} expresses the counts $\J(v)$ of Gieseker semistable sheaves of rank 0 and dimension 2 in terms of those which are semistable in the large volume chamber.\medskip

Our main technique for finding walls of instability is the Bogomolov-Gieseker conjecture of \cite{BMT, BMS}. This is now proved for many 3-folds, including some Calabi-Yau 3-folds \cite{BMS, Ko20, Li, MP}. The restricted set of weak stability conditions handled in \cite{Ko20, Li} are sufficient for our purposes, as we check carefully on the quintic 3-fold in Section \ref{Chunyi}. 
The MNOP conjecture is also proved for the quintic in \cite{PP}. So all DT invariants of the quintic 3-fold are given by universal formulae in its Gromov-Witten invariants.

%

\subsection*{Outlook}
Having dealt with the counting of sheaves, one could try to extend our result to counts of arbitrary Bridgeland semistable complexes of sheaves. After a possible shift to ensure rank $\ge0$ this reduces to another wall crossing problem --- assuming the space of stability conditions is connected --- to reach stability conditions dominated by the weak stability conditions considered in this paper and thus return to the problem of counting sheaves.

Applying Joyce's new wall crossing formula for Fano 3-folds \cite[Theorem 7.69]{Jo} to the results of Section \ref{shvs} should hopefully result in a Fano version of Theorem \ref{1} with insertions.

Kontsevich and Soibelman's refinements of $\J(v)$ \cite{KS1, KS2} should be well defined now their integrality conjecture has been proved \cite{BenSven} and orientation data has been shown to exist on Calabi-Yau 3-folds \cite{JU}. By using their wall crossing formula \cite{KS1} in place of \cite{JS} it should be possible to prove an analogue of Theorem \ref{1} for these invariants too.

Theorem \ref{1} is an abstract existence theorem for a universal expression $F$; it is not a practical route to concrete formulae. For explicit formulae for related invariants in special cases see \cite{Ob} (for reduced DT invariants of $K3\times E$), \cite[Equation 1.6]{FMR}, \cite[Equation 5.30]{DNPZ}, \cite{AK} (for K-theoretic counts of quotients of a trivial bundle on noncompact toric Calabi-Yau 3-folds) and \cite{TodaBG,F21} (for explicit forms of Theorem \ref{2} expressing counts of dimension 2 sheaves --- with restricted Chern classes ---
in terms of curve counts).

\subsection*{Acknowledgements}
We are grateful to Arend Bayer, Tom Bridgeland, Dominic Joyce, Davesh Maulik, Rahul Pandharipande and Yukinobu Toda for generous help and discussions about DT theory and stability conditions over many years. We also thank an anonymous referee.

We acknowledge the support of an EPSRC postdoctoral fellowship EP/T018658/1, an EPSRC grant EP/R013349/1 and a Royal Society research professorship.

\setcounter{tocdepth}{1}
\tableofcontents

\section{Weak stability conditions}\label{weak}
Let $(X, \cO(1))$ be a smooth polarised complex projective threefold with bounded derived category of coherent sheaves $\cD(X)$ and Grothendieck group $K(\mathrm{Coh}(X))$. Dividing by the kernel of the Mukai pairing gives the numerical Grothendieck group
\beq{Kdef}
K(X)\ :=\ \frac{K(\mathrm{Coh}(X))}{\ker\chi(\ \ ,\ \ )}\,.
\eeq
Notice that $K(X)$ is torsion-free, isomorphic to its image in $H^*(X,\Q)$ under the Chern character. Denoting $H=c_1(\cO_X(1))$, for any $v\in K(X)$ we set\footnote{Note $\ch\_H$ meant something different in \cite{FT3}.}
\beqa
\ch\_H(v) &:=& \Big(\!\ch_0(v),\ \tfrac1{H^3}\ch_1(v).H^2,\ \tfrac1{H^3}\ch_2(v).H,\ \tfrac1{H^3}\ch_3(v)\;\!\Big)\ \in\ \Q^4,\\
\ch_H^{\le2}(v) &:=& \Big(\!\ch_0(v),\ \tfrac1{H^3}\ch_1(v).H^2,\ \tfrac1{H^3}\ch_2(v).H\;\!\Big)\ \in\ \Q^3.
\eeqa

Scaling the usual definition by $H^3$, we define the $\mu\_H$-slope of a coherent sheaf $E$ to be
$$
\mu\_H(E)\ :=\ \left\{\!\!\!\begin{array}{cc} \frac{\ch_1(E).H^2}{\ch_0(E)H^3} & \text{if }\ch_0(E)\ne0, \\
+\infty & \text{if }\ch_0(E)=0. \end{array}\right.
$$
Associated to this slope every sheaf $E$ has a Harder-Narasimhan filtration. Its graded pieces have slopes whose maximum we denote by $\mu_H^+(E)$ and minimum by $\mu_H^-(E)$.

For any $b \in \mathbb{R}$, let $\cA_{\;b}\subset\cD(X)$ denote the abelian category of complexes
	\begin{equation}\label{Abdef}
	\Ab\ =\ \big\{E^{-1} \xrightarrow{\ d\ } E^0 \ \colon\ \mu_H^{+}(\ker d) \leq b \,,\  \mu_H^{-}(\cok d) > b \big\}. 
	\end{equation}
In particular, setting $\ch^{bH}(E):=\ch(E)e^{-bH}$, each $E\in\cA_{\;b}$ satisfies
\beq{65}
\ch_1(E).H^2-bH^3\ch_0(E)\=\ch_1^{bH}(E).H^2\=\ch_1^{bH}(\cok d)-\ch_1^{bH}(\ker d)\ \ge\ 0,
\eeq
with $\ge0$ replaced by $>0$ when $\cok d=\cH^0(E)$ has dimension $\ge2$.
By \cite[Lemma 6.1]{Br} $\cA_{\;b}$ is the heart of a t-structure on $\cD(X)$. We denote its positive cone by
\beq{CAb}
C(\Ab)\ :=\ \Big\{\sum\nolimits_ia_i[E_i]\ \colon\ a_i\,\in\,\N,\ E_i\,\in\,\Ab\Big\}\ \subset\ K(X).
\eeq
For any $w>\frac12b^2$, we have on $\cA_{\;b}$ the slope function
\begin{equation}\label{noo}
\nubw(E)\ =\ \left\{\!\!\!\begin{array}{cc} \frac{\ch_2(E).H - w\ch_0(E)H^3}{\ch_1^{bH}(E).H^2}
 & \text{if }\ch_1^{bH}(E).H^2\ne0, \\
+\infty & \text{if }\ch_1^{bH}(E).H^2=0. \end{array}\right.
\end{equation}
By \cite{BMT}\footnote{We use notation from \cite{FT2}; in particular the rescaling \cite[Equation 6]{FT2} of \cite{BMT}'s slope function.} $\nubw$ defines a Harder-Narasimhan filtration on $\cA_{\;b}$, and so a \emph{weak stability condition} on $\cD(X)$.

\begin{Def}\label{df1}
Fix $w>\frac12b^2$. Given an injection $F\into E$ in $\cA_{\;b}$ we call $F$ a \emph{destabilising subobject of} $E$ if and only if
\beq{seesaw1}
\nubw(F)\ \ge\ \nubw(E/F),
\eeq
and \emph{strictly destabilising} if $>$ holds. 
We say $E\in\cD(X)$ is $\nubw$-(semi)stable if and only if
\begin{itemize}
\item $E[k]\in\cA_{\;b}$ for some $k\in\Z$, and
\item 
$E[k]$ contains no (strictly) destabilising subobjects.
\end{itemize}
\end{Def}
It is important to note that under these conventions an object can be \emph{destabilised} even if it is \emph{semistable} --- in fact, if and only if it is strictly semistable. That is, our ``destabilised" is what others might call ``semi-destabilised" (as we did on page 1).

Also note that we cannot replace \eqref{seesaw1} by $\nubw(F)\ge\nubw(E)$; this is implied by \eqref{seesaw1} but does not imply it. So for instance the sequence $I_p\into\cO_X\onto\cO_p$, for $p$ a point of $X$, does not destabilise $\cO_X$ even though $\nubw(I_p)=\nubw(\cO_X)$.

\begin{Rem}\label{heart}
Given $(b,w) \in \mathbb{R}^2$ with $w> \frac12b^2$, the argument in \cite[Propostion 5.3]{Br.stbaility} describes $\cA_{\;b}$. It is the extension-closure of the set of $\nubw$-stable two-term complexes $E = \{E^{-1} \to E^0\}$ in $\cD(X)$ satisfying the following conditions on the denominator and numerator of $\nubw$ \eqref{noo}:
\begin{enumerate}
    \item $\ch_1^{bH}(E).H^2 \geq 0$, and
    \item $\ch_2(E).H - w\ch_0(E)H^3 \geq 0$ if $\ch_1^{bH}(E).H^2 = 0$. 
\end{enumerate}
\end{Rem}

By \cite[Theorem 3.5]{BMS} any $\nubw$-semistable object $E \in \mathcal{D}(X)$ satisfies
\beq{BOG}
    \Delta_H(E)\ :=\ \left(\ch_1(E).H^2\right)^2 -2(\ch_2(E).H)\ch_0(E)H^3\ \geq\ 0.
\eeq
Therefore, if we plot the $(b,w)$-plane simultaneously with the image of the projection map
\begin{eqnarray*}
	\Pi\colon\ K(X) \smallsetminus \big\{E \colon \ch_0(E) = 0\big\}\! &\longrightarrow& \R^2, \\
	E &\ensuremath{\shortmid\joinrel\relbar\joinrel\rightarrow}& \!\!\bigg(\frac{\ch_1(E).H^2}{\ch_0(E)H^3}\,,\, \frac{\ch_2(E).H}{\ch_0(E)H^3}\bigg),
\end{eqnarray*}
as in Figure \ref{projetcion}, then $\nubw$-semistable objects $E$ lie outside the
 open set
\begin{equation}\label{Udef}
U\ :=\ \Big\{(b,w) \in \mathbb{R}^2 \colon w > \tfrac12b^2  \Big\}
\end{equation}
while $(b,w)$ lies inside $U$.
\begin{figure}[h]
	\begin{centering}
		\definecolor{zzttqq}{rgb}{0.27,0.27,0.27}
		\definecolor{qqqqff}{rgb}{0.33,0.33,0.33}
		\definecolor{uququq}{rgb}{0.25,0.25,0.25}
		\definecolor{xdxdff}{rgb}{0.66,0.66,0.66}
		
		\begin{tikzpicture}[line cap=round,line join=round,>=triangle 45,x=1.0cm,y=0.9cm]
		
		\draw[->,color=black] (-4,0) -- (4,0);
		\draw  (4, 0) node [right ] {$b,\,\frac{\ch_1\!.\;H^2}{\ch_0H^3}$};


		\fill [fill=gray!40!white] (0,0) parabola (3,4) parabola [bend at end] (-3,4) parabola [bend at end] (0,0);
		
		\draw  (0,0) parabola (3.1,4.27); 
		\draw  (0,0) parabola (-3.1,4.27); 
		\draw  (3.8 , 3.6) node [above] {$w= \frac{b^2}{2}$};

		\draw[->,color=black] (0,-.8) -- (0,4.7);
		\draw  (1, 4.1) node [above ] {$w,\,\frac{\ch_2\!.\;H}{\ch_0H^3}$};
		
		\draw [dashed, color=black] (-2.3,1.5) -- (-2.3,0);
		\draw [dashed, color=black] (-2.3, 1.5) -- (0, 1.5);
		\draw [color=black] (-2.6, 1.36) -- (1.3, 3.14);
		
		\draw  (-2.8, 1.8) node {$\Pi(E)$};
		\draw  (-1, 3) node [above] {\Large{$U$}};
		\draw  (0, 1.5) node [right] {$\frac{\ch_2(E).H}{\ch_0(E)H^3}$};
		\draw  (-2.3 , 0) node [below] {$\frac{\ch_1(E).H^2}{\ch_0(E)H^3}$};
		\begin{scriptsize}
		\fill (0, 1.5) circle (2pt);
		\fill (-2.3,0) circle (2pt);
		\fill (-2.3,1.5) circle (2pt);
		\fill (1,3) circle (2pt);
		\draw  (1.2, 2.96) node [below] {$(b,w)$};
		
		\end{scriptsize}
		
		\end{tikzpicture}
		
		\caption{$(b,w)$-plane and the projection $\Pi(E)$ of a $\nubw$-semistable object $E\in\Ab$ with $\ch_0(E)<0$}
		
		\label{projetcion}
		
	\end{centering}
\end{figure}
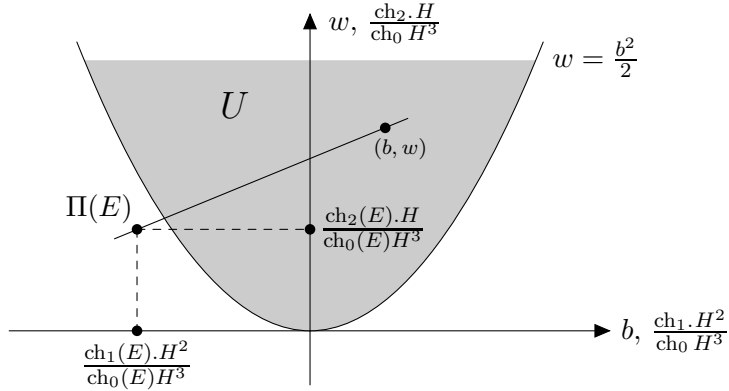
Shifting if necessary we may assume $E\in\Ab$, then by \eqref{65} we see $\Pi(E)$ lies on or to the left of the vertical line through $(b,w)$ if $\ch_0(E)<0$, to the right if $\ch_0(E)>0$, and at infinity if $\ch_0(E)=0$. The slope $\nubw(E)$ of $E$ is the gradient of the line connecting $(b,w)$ to $\Pi(E)$ (or $\ch_2(E).H\big/\!\ch_1(E).H^2$ if $\rk(E)=0$).

Any object $E\in\cD(X)$ gives the space of weak stability conditions a wall and chamber structure by \cite[Proposition 12.5]{BMS}, as rephrased in \cite[Proposition 4.1]{FT1} for instance.


\begin{Prop}[\textbf{Wall and chamber structure}]\label{prop. locally finite set of walls}
	Fix $v\in K(X)$ with $\Delta_H(v)\ge0$ and $\ch_H^{\le2}(v)\ne0$. There exists a set of lines $\{\ell_i\}_{i \in I}$ in $\mathbb{R}^2$ such that the segments $\ell_i\cap U$ (called ``\emph{walls of instability}") are locally finite and satisfy 
	\begin{itemize*}
	    \item[\emph{(}a\emph{)}] If $\ch_0(v)\ne0$ then all lines $\ell_i$ pass through $\Pi(v)$.
	    \item[\emph{(}b\emph{)}] If $\ch_0(v)=0$ then all lines $\ell_i$ are parallel of slope $\frac{\ch_2(v).H}{\ch_1(v).H^2}$.
	   		\item[\emph{(}c\emph{)}] The $\nubw$-(semi)stability of any $E\in\cD(X)$ of class $v$ is unchanged as $(b,w)$ varies within any connected component (called a ``\emph{chamber}") of $U \smallsetminus \bigcup_{i \in I}\ell_i$.
		\item[\emph{(}d\emph{)}] For any wall $\ell_i\cap U$ there is a map $f\colon F\to E$ in $\cD(X)$ such that
\begin{itemize}
\item for any $(b,w) \in \ell_i \cap U$, the objects $E,\,F$ lie in the heart $\cA_{\;b}$,
\item $E$ is $\nubw$-semistable of class $v$ with $\nubw(E)=\nubw(F)=\,\mathrm{slope}\,(\ell_i)$ constant on the wall $\ell_i \cap U$, and
\item $f$ is an injection $F\into E $ in $\cA_{\;b}$ which strictly destabilises $E$ for $(b,w)$ in one of the two chambers adjacent to the wall $\ell_i$.
\hfill$\square$

\end{itemize} 
	\end{itemize*} 
\end{Prop}

We can improve on this local finiteness by showing we have finiteness as $w\to\infty$. This gives, for each fixed $v\in K(X)$, a \emph{large volume chamber} $\subseteq U$ in which there are no walls for $v$, so the $\nubw$-(semi)stability of objects of class $v$ is independent of $w\gg0$. It is the subset of $U$ above the uppermost walls described below, or all of $U$ if there are no walls.

When $\rk(v)\ne0$ we temporarily denote the half of $U$ to the left or right of $\Pi(v)$ by
$$
U_<\,:=\,\Big\{(b,w)\in U\,\colon\,b\,<\,\tfrac{\ch_1(v).H^2}{\ch_0(v)H^3}\Big\}\ \text{ and }\  U_>\,:=\,\Big\{(b,w)\in U\,\colon\, b\,>\,\tfrac{\ch_1(v).H^2}{\ch_0(v)H^3}\Big\}.
$$
\begin{Prop}[\textbf{Large volume chamber}]\label{large}
Fix $v\in K(X)$ with $\Delta_H(v)\ge0$.

If $\rk(v)=0$ (respectively $\rk(v)\ne0$) and there is a wall of instability for $v$ in $U$ (respectively $U_<$ or $U_>$), then there is an uppermost such wall.
\end{Prop}

\begin{proof}
By the results of Proposition \ref{prop. locally finite set of walls} it is sufficient to show that for a fixed $b_0\in\Q$ not equal to $\frac{\ch_1(v).H^2}{\ch_0(v)H^3}$ there are only finitely many points $(b_0,w)$ which lie on walls of instability for the class $v$. 

A point $(b_0,w)$ on a wall of instability for $v$ gives a decomposition $v=v_1+v_2$ with
\beq{inq}
0\ \le\ \ch_1^{b_0H}(v_i).H^2\ \le\ \ch_1^{b_0H}(v).H^2\ \text{ and }\ 
0\ \le\ \Delta_H(v_i)\ <\ \Delta_H(v),
\eeq
the first by \eqref{Abdef} and the second by \cite[Corollary 3.10]{BMS} or \cite[Lemma 3.2]{FT3}. We will show this means there are only finitely many points $\(\!\ch_0(v_i),\,\ch_1^{bH}(v_i).H^2,\,\ch_2^{bH}(v_i).H\)\in\Q^3$ corresponding to such decompositions. There are therefore only finitely many $\Pi(v_i)\in\Q^2$, which by Proposition \ref{prop. locally finite set of walls} means only finitely many walls, as required. \medskip

Since $b_0\in\frac1N\Z$ for some $N\in\N$, the bounds \eqref{inq} imply there are only finitely many values $\ch_1^{b_0H}(v_i).H^2\in\frac1N\Z$ can take, and only finitely many integers $\Delta_H(v_i)$ can take. Furthermore, rewriting $\Delta_H$ as
$$
\Delta_H\=\(\!\ch_1^{b_0H}\!\!.H^2\)^2-2\(\!\ch_2^{b_0H}\!\!.H\)\ch_0H^3,
$$
we see that
\beq{fint}
(\ch_2^{b_0H}(v_i).H)\ch_0(v_i)\,\text{ takes only finitely many values.}
\eeq
If $\ch_0(v_i)=0$ for $i=1,2$ then both slopes $\nubw(v_i)$ are constant in $w$, contradicting the existence of a wall. So without loss of generality $\ch_0(v_1)\ne0$, which by \eqref{fint} and bounded denominators means that $\ch_2^{b_0H}(v_1).H$ takes only finitely many values. The same therefore also applies to 
$$
\ch_2^{b_0H}(v_2).H\=\ch_2^{b_0H}(v).H-\ch_2^{b_0H}(v_1).H.
$$
Finally if $\ch_2^{b_0H}(v_i)=0$ for $i=1,2$ then a simple calculation shows both $\Pi(v_i)$ lie on the line $w=b_0b-\frac12b_0^2$. Therefore so does $\Pi(v)$, so by 
Proposition \ref{prop. locally finite set of walls} this is the line of instability. But it is tangent to $\partial U$ at $(b_0,\frac12b_0^2)$, so does not pass through $U$. Thus $\ch_2^{b_0H}(v_1)\ne0$ without loss of generality and by \eqref{fint} there are only finitely many values for $\ch_0(v_1)$ and $\ch_0(v_2)=\ch_0(v)-\ch_0(v_1)$.
\end{proof}

In this paper, we always assume $X$ satisfies the conjectural Bogomolov-Gieseker inequality of Bayer-Macr\`i-Toda \cite{BMT}.  In the form of \cite[Conjecture 4.1]{BMS}, rephrased in terms of the rescaling \cite[Equation 6]{FT2}, it is the following.

\begin{Con}[\textbf{Bogomolov-Gieseker inequality}]\label{conjecture}
For any $(b,w)\in U$ and $\nubw$-semistable $E\in\cD(X)$, we have the inequality
\begin{equation}\label{quadratic form}
    B_{b, w}(E)\ :=\ (2w-b^2)\Delta_H(E) + 4\big(\!\ch_2^{bH}(E).H\big)^2 -6\(\!\ch_1^{bH}(E).H^2\)\ch_3^{bH}(E)\ \geq\ 0.
\end{equation}
\end{Con}

Multiplying out and cancelling we find that $B_{b, w}$ is actually linear in $(b,w)$:
\beq{boglinear}
\tfrac12B_{b, w}(E)\=\(C_1^2-2C_0C_2\)w+\(3C_0C_3-C_1C_2\)b+(2C_2^2-3C_1C_3),
\eeq
where $C_i:=\ch_i(E).H^{3-i}$.
The coefficient of $w$ is $\ge0$ by \eqref{BOG}. When it is $>0$ the Bogomolov-Gieseker inequality \eqref{quadratic form} says that $E$ can be $\nubw$-semistable only above the line $\ell_f(E)$ defined by the equation $B_{b, w}(E)=0$. When $\ch_0(E)\neq0\neq\ch_1(E).H^2$ we can rearrange to see $\ell_f(E)$ is the line through the points $\Pi(E)$ and 
\begin{equation}\label{pi'}
\Pi'(E)\ :=\, \left(\frac{2\ch_2(E).H}{\ch_1(E).H^2}\,, \ \frac{3\ch_3(E)}{\ch_1(E).H^2}\right).    
\end{equation}

\section{Walls}\label{shvs}
Throughout the whole of this Section we fix a rank 0 class $\v\in K(X)$ with
\beq{vdef}
\ch\_H(\v)\=(0,\,c,\,s_0,\,d_0),
\eeq
where $c>0$. Then we pick $n_0\gg0$ as follows. There are finitely many explicit inequalities of the form $O\(n_0^i\)<O(n_0^{i+1})$ in this Section, starting at \eqref{slope} and ending with \eqref{lastn}. We fix
\beq{nzero}
n_0\=n_0(c,s_0,d_0)\ \gg\ 0
\eeq
sufficiently large that all of these inequalities hold. 
Then set $\vno:=\v-[\cO_X(-n_0)]$, so
\begin{equation}\label{class vn}
\ch\_H(\vno)\=\Big(\!-1, \ c+n_0,\ s_0-\tfrac12n_0^2 , \ d_0 +\tfrac16n_0^3\Big).
\end{equation}
This Section studies stability conditions $\nubw$ with $b$ strictly to the \emph{right} of $\Pi(\vno)$, i.e.
\beq{right}
b\ >\ -(c+n_0),\quad\text{so that}\quad \nubw(\vno)\ <\ +\infty.
\eeq

By \eqref{pi'} all walls for class $\vno$ lie on or above the line $\ell_f=\ell_f(\vno)$ which passes through 
\begin{equation*}
\Pi(\vno) = \Big(\!-(n_0+c) ,\ -s_0+\tfrac12n_0^2 \Big) \quad\text{and}\quad \Pi'(\vno) = \left(\frac{2s_0-n_0^2}{n_0+c}\,, \ \frac{3d_0 +\frac12n_0^3}{n_0+c}  \right).
\end{equation*} 	
The equation of $\ell_f$ then works out to be 
\beq{lfeq}
4w \= \left(\!-n_0 + \frac{n_0(6s_0 +c^2) +4(cs_0+3d)}{2cn_0 +c^2+2s_0}  \right)\!b + n_0^2
+ \frac{n_0^2(6s_0 -c^2) + 4(3dn_0 + 3dc-2s_0^2)}{2cn_0 +c^2+2s_0}\,.
\eeq 
Since $c\in\frac1{H^3}\Z$ this gives, for $n_0\gg0$,
\begin{equation}\label{slope}
\mathrm{slope}\;(\ell_f)\ >\ -\frac{n_0}{4} -|s_0|H^3.
\end{equation}
Setting $w=\frac12b^2$ in \eqref{lfeq} gives a quadratic equation with roots the values $b_1^f<b_2^f$ of $b$ at the two intersection points $\ell_f\cap\partial U$,
\beq{b1b2}
b_1^f\=-n_0 + \tfrac13c + O\(\tfrac1{n_0}\) \quad \text{and} \quad b_2^f\=\tfrac12n_0 -\tfrac1{12}c + \tfrac{3s_0}{2c} + O\(\tfrac1{n_0}\).
\eeq
The large distance $\frac32n_0+O(1)$ between these points enables us to restrict the form of certain $\nubw$-semistable objects, such as those of class $\vno$, and their semistable factors. So fix data
\begin{itemize}
\item a $\nubw$-semistable object $E\in\Ab$ of rank $-1$ and $\ch_1(E).H^2\le(n_0+c)H^3$,
\item a sequence $E'\into E\onto E''$ in $\Ab$ with all objects of the same $\nubw$-slope, where
\item $(b,w)$ lies to the right of $\Pi(E)$, and the line $\ell$ joining them is on or above $\ell_f\cap U$.\footnote{Throughout the paper this phrase will mean that $w$ is larger on $\ell$ than on $\ell_f$ whenever $b\in\big[b_1^f,b_2^f\big]$ \eqref{b1b2}.}
\end{itemize}
We allow either of $E',E''$ to be 0, in which case we could take the other to be \mbox{$E$ of class $\vno$.}

\begin{Lem}\label{lem.destabilising objects}
There is an ordering $E_0,E_1$ of $E',E''$ such that
\begin{itemize}
    \item $E_0$ is a sheaf of rank $0$; if it is nonzero then $\ch_1(E_0).H^2>0$, 
    \item $\rk\!\(\cH^{-1}(E_1)\)=1,\ \rk\!\(\cH^0(E_1)\)=0$ and $\ch_1(E_1).H^2\ge\(n_0-\frac c3\)H^3$.
\end{itemize}
\end{Lem}

\begin{proof}
Since $\rk(E)=-1$, one of the objects $E',E''$ has rank$\,<0$; call it $E_1$. The other $E_0$ has rank$\,\ge0$.
Set $r_i:=\rk\!\(\cH^{-1}(E_i)\)\ge0$. Since $\ell$ lies on or above $\ell_f\cap U$, it lies inside $U$ for $b$ in the interval $(b_1^f,b_2^f)$. Thus $E_i\in\Ab$ by Proposition \ref{prop. locally finite set of walls}. By \eqref{Abdef} this gives
	\begin {eqnarray} \label{chr1}
	\ch_1\!\(\cH^{-1}(E_i)\).H^2 &\le& \Big[\!-n_0+\tfrac13c+O\(\tfrac1{n_0}\)\Big]r_iH^3 \\
	\text{and}\quad \label{chr2}
	\ch_1\!\(\cH^0(E_i)\).H^2 &\ge& \Big[\tfrac12n_0-\tfrac1{12}c+\tfrac{3s_0}{2c}+O\(\tfrac1{n_0}\)\Big](\ch_0(E_i)+r_i)H^3.
	\end{eqnarray}
	In particular, subtracting gives
	\beq{devon}
	\ch_1(E_i).H^2\ \ge\ \Big[\tfrac12n_0+O(1)
	\Big]\ch_0(E_i)H^3+\Big[\tfrac32n_0+O(1)
	\Big]r_iH^3.
	\eeq	
	Adding over $i=0,1$ gives $\ch_1(E).H^2$, so
	\beq{chr3}
	(n_0 +c)H^3\ \ge\ -\Big[\tfrac12n_0+O(1)\Big]H^3+\Big[\tfrac32n_0+O(1)\Big](r_0+r_1)H^3.
	\eeq
	Taking the coefficient of $n_0$ gives $1\ge r_0+r_1$. But $r_1\ge1$, so in fact $r_0=0$ and $r_1=1$. From the exact sequence
	\beq{LES}
	0\to\cH^{-1}(E')\To\cH^{-1}(E)\To\cH^{-1}(E'')\To\cH^0(E')\To\cH^0(E)\To\cH^0(E'')\to0
	\eeq
we deduce $\rk\!\(\cH^{-1}(E)\)=1$ while $\cH^0(E_0),\,\cH^0(E_1),$ $\cH^0(E)$ and $\cH^{-1}(E_0)$ are all sheaves of rank 0. By the definition \eqref{Abdef} of $\Ab$ the latter implies $\cH^{-1}(E_0)=0$, so $E_0$ is a sheaf of rank 0. Moreover $\nubw(E_0)=\nubw(E)<+\infty$, so $\ch_1(E_0).H^2>0$ if $E_0\ne0$.\medskip

Finally substituting $i=1,\ r_i=1$ and $\ch_0(E_1)+r_i=0$ into (\ref{chr1}, \ref{chr2}) gives
\[
\ch_1(E_1).H^2\=\big[\!\;\ch_1\!\(\cH^0(E_1)\)-\ch_1\!\(\cH^{-1}(E_1)\)\big].H^2\ >\ \(n_0-\tfrac13c\)H^3+O\(\tfrac1{n_0}\).
\]
Since $n_0,\,\ch_1(E_1).H^2 \in \Z$ and $c\in\frac1{H^3}\Z$, we conclude $\ch_1(E_1).H^2 \geq (n_0-\tfrac13c\)H^3$.
\end{proof}

It turns out $E_0$ is already semistable in its large volume chamber.

\begin{Lem}\label{rk0ss}
There are no walls for $E_0$ which lie on or above $\ell_f\cap U$.  
\end{Lem}

\begin{proof}
Lemma \ref{lem.destabilising objects} gives $\ch_1(E_1).H^2\ge\(n_0-\frac c3\)H^3$, so
\beq{E0bd}
\ch_1(E_0).H^3\ \le\ \tfrac43cH^3.
\eeq
Given a wall of instability $\ell_0 \cap U$ for $E_0$ which lies along or above $\ell_f\cap U$, the difference in $b$-values $b_1<b_2$ of the points of $\ell_0\cap\partial U$ satisfies $b_2-b_1\ge b_2^f-b_1^f=\frac32n_0+O(1)$ by \eqref{b1b2}.
Repeating the arguments (\ref{chr1}, \ref{chr2}, \ref{devon}, \ref{chr3}) for destabilising factors $F_i$ of $E_0$, \eqref{chr3} becomes
$$
\ch_1(E_0).H^2\ \ge\ b_2\rk(E_0)H^3+(b_2-b_1)(r_0+r_1)H^3,
$$
where $\rk(E_0)=0$ and $r_i:=\rk\(\cH^{-1}(F_i)\)$. So by \eqref{E0bd},
\beq{tick}
\tfrac43cH^3\ \ge\ \(\tfrac32n_0+O(1)\)(r_0+r_1)H^3.
\eeq
 Thus $r_0=0=r_1$ and both $F_i$ have rank 0. But then their $\nubw$ slopes \eqref{noo} --- and that of $E_0$ --- are constant in $(b,w)$, contradicting the existence of the wall $\ell_0$.
\end{proof}

As we cross walls of instability for $\vno$ above $\ell_f$ we will have to deal with semistable factors of various types, and their own wall crossing. One type will be classes
$$
w_n\=\w-\,[\cO_X(-n )]\ \text{ which are close to }\  \vno\=\v-[\cO_X(-n_0)]
$$
in an appropriate sense. In fact it is important to note we will allow fractional $n\in\frac1{H^3}\Z$ and $\w\in K(X)\otimes\Q$ so long as $w_n\in K(X)$ is integral.

\begin{Def}\label{close} We say that $w_n\in K(X)$ is \emph{close to} $\vno$ if
\begin{itemize}
    \item $\Pi(w_n)$ lies on or above $\ell_f$ and
    \item $\ch\_H(w_n)=(0,c,s,d)-\(1,-n,\frac12n^2,-\frac16n^3\)$,
    \end{itemize}
    where $n\in\frac1{H^3}\Z$ and $s,d$ satisfy the following bounds.\vspace{1mm}
\begin{itemize}
    \item $n\in\big[n_0-\frac c3,\,n_0\big]$. That is, if we set $\delta_n:=n_0-n\in\frac1{H^3}\Z$ then $\delta_n\in\big[0,\frac c3\big]$.\vspace{1.6mm}
    \item $-n_0\delta_n-|s_0| \,\le\,  s \,\le\, -\frac34n_0\delta_n+ \(c +|s_0|H^3\)\delta_n +s_0$.\vspace{1mm}
    \item $d\,\ge\,\frac{15}{32}n_0^2\delta_n -n_0\delta_n\(|s_0|H^3 +c\) -\delta_n\(s_0H^3\)^2 + d_0$.\vspace{1mm}
\end{itemize} 
In particular setting $\delta_n=0$ shows that $\vno$ is close to $\vno$, which is good.
\end{Def}


%
For such a class we let $\ell(w_n)$ be the line parallel to $\ell_f$ through the point $\Pi(w_n)$ as in Figure \ref{figUwm}. We will only ever need to understand walls of instability inside the region
$$
U(w_n)\ :=\ \Big\{(b,w)\in U\ \colon\, (b,w)\text{ is on or above }\ell(w_n)\text{ and to the right of }\Pi(w_n)\Big\}.
$$\vspace{-5mm}
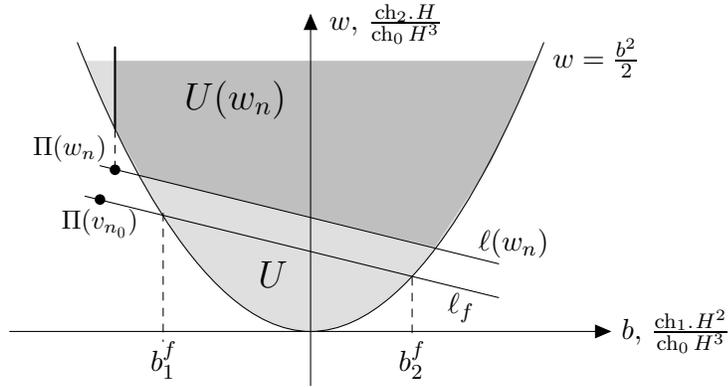
\begin{figure}[h]
	\begin{centering}
		\definecolor{zzttqq}{rgb}{0.27,0.27,0.27}
		\definecolor{qqqqff}{rgb}{0.33,0.33,0.33}
		\definecolor{uququq}{rgb}{0.25,0.25,0.25}
		\definecolor{xdxdff}{rgb}{0.66,0.66,0.66}
		
		\begin{tikzpicture}[line cap=round,line join=round,>=triangle 45,x=1cm,y=0.9cm]
		
		\draw[->,color=black] (-4,0) -- (4,0);
		\draw  (4, 0) node [right ] {$b,\,\frac{\ch_1\!.\;H^2}{\ch_0H^3}$};
		
		\fill [fill=gray!25!white] (0,0) parabola (3,4) parabola [bend at end] (-3,4) parabola [bend at end] (0,0);
		
	    \fill[fill=gray!50!white](-2.6, 4)--(-2.6,3)--(-2.25,2.3)--(1.65,1.25)--(2.4,2.5)--(3,4); 
	   
		\draw  (0,0) parabola (3.1,4.27); 
		\draw  (0,0) parabola (-3.1,4.27); 
		\draw  (3.8 , 3.6) node [above] {$w= \frac{b^2}{2}$};
	
		\draw[->,color=black] (0,-.8) -- (0,4.7);
		\draw  (1, 4.1) node [above ] {$w,\,\frac{\ch_2\!.\;H}{\ch_0H^3}$};
		
		\draw [color=black] (-3, 2) -- (2.5, .5);
		\draw [color=black] (-2.8, 2.45) -- (2.5, 1);
	    \draw [dashed] (-1.96, 1.75) -- (-1.96,-0.15);		
	    \draw [dashed] (1.35,.8) -- (1.35,-0.1);
	    \draw [dashed] (-2.6,2.4) -- (-2.6,3);
	    \draw [thick] (-2.6,3) -- (-2.6,4.2);
	    
	    \draw  (-1.95, 0) node[below]{$b_1^f$};
	    \draw  (1.35, 0) node[below]{$b_2^f$};
	    \draw  (-2.8, 1.6) node {{\small$\Pi(\vno)$}};
	    \draw  (-3.2, 2.7) node {{\small$\Pi(w_n)$}};
		\draw  (2, .35) node {$\ell_f$};
		\draw  (2.7, 1.3) node {$\ell(w_n)$};
        \draw  (-1, 3) node [above] {\Large{$U(w_n)$}};
        \draw  (-.5, .5) node [above] {\Large{$U$}};
		\fill (-2.6,2.39) circle (2pt);
		\fill (-2.8,1.95) circle (2pt);

		\end{tikzpicture}
		
		\caption{The line $\ell(w_n)$ and the region $U(w_n)\subset U$}
		\label{figUwm}
	\end{centering}
\end{figure}

%
We now analyse the possible wall crossing of objects of class $\vno$ and, more generally, classes $w_n$ close to $\vno$. So for the rest of this Section we fix the following data,
\begin{itemize}
    \item a class $w_n$ close to $\vno$ in the sense of Definition \ref{close}, 
    \item a weak stability condition $(b,w)\in U(w_n)$,
    \item a $\nubw$-semistable object $E$ of class $w_n$, and
    \item a $\nubw$-destabilising sequence $E'\into E\onto E''$ in $\Ab$.
\end{itemize}
We then let $\ell$ denote the line through $(b,w)$ and $\Pi(w_n)$. Since $\ch_1(w_n).H^2=(n+c)H^3\le(n_0+c)H^3$ we may apply Lemma \ref{lem.destabilising objects}. This renames $E',E''$ (in some order) as $E_0,E_1$ of ranks $0,-1$ respectively.

Since $\ell$ lies on or above $\ell_f\cap U$, and all possible walls for $E_0$ are parallel to $\ell$, we conclude from Lemma \ref{rk0ss} that there are no walls for $E_0$ on or above $\ell\cap U$.

We write the classes of the destabilising factors $E_0,\,E_1$ as
\beq{vec}
\ch\_H(E_0)\=\(0,\,c-c',\,s-s',\,d-d'\), \quad \ch\_H(E_1)\=\(\!\;-1,\,n +c',\,-\tfrac12n^2+s', \,\tfrac16n^3+d'\),
\eeq
for some $c', s', d' \in \mathbb{Q}$ with $c-c'>0$. The key result in this paper is the following.

\begin{Prop}\label{c' negative}
	Either $c'\in(0,c)$ or $c'= 0$ and $\delta_n = 0$. 
\end{Prop}

\begin{proof}
	Assume $c' \leq 0$. By Lemma \ref{lem.destabilising objects}, $n+c'\ge n_0-\frac13c$, so we get the bounds
\beq{del}
-\delta_n +c'\ \geq\  -\tfrac13c \quad\text{and}\quad c'\,\in\,\big[\!-\tfrac13c,\,0\big]\cap\tfrac1{H^3}\Z.
\eeq

We can also bound $s'$. By Proposition \ref{prop. locally finite set of walls} the wall $\ell$ passes through $\Pi(E_1)$ and $\Pi(w_n)$, so has equation
	\beq{ell}
	w-\tfrac12n^2 + s\=\frac{s-s'}{c-c'}(b+n+c).
	\eeq
Since $\ell$ lies above or on $\ell(w_m)$, which has slope $-\frac{n_0}{4} + O(1)$ by \eqref{lfeq}, we get the bound
	\begin{equation}\label{s-1}
	s'\ \leq\ s+ (c-c')\left( \frac{n_0}{4} +O(1)\right). 
	\end{equation}
	 The Bogomolov inequality \eqref{BOG} for $E_1$ is $0\le\Delta_H(E_1)=(n+c')^2+2\(s'-\frac12n^2\)$, which is
$$
s'\ \ge\ -\(n+\tfrac12c'\)c'\ \stackrel{\eqref{del}}\ge\ 0.
$$
Combined with \eqref{s-1}, \eqref{del} and the bound $s\le s_0$ from Definition \ref{close} this bounds
\beq{s'}
0\ \le\ s'\ \le\ \tfrac13cn_0+O(1).
\eeq\smallskip

Next we apply the Bogomolov-Gieseker inequality to $\ch\_H(E_0)$ and $\ch\_H(E_1)$ \eqref{vec}. First consider $E_0$. By Lemma \ref{lem.destabilising objects} it is a \emph{sheaf of rank $0$} which is $\nubw$-semistable on $\ell\cap U$, so \cite[Lemma B.3]{FT3} applies. There we showed it is $\nubw$-semistable at a judicious point $(b,w)$ of $U$ where the Bogomolov-Gieseker inequality \eqref{quadratic form} bounds its $\ch_3$ by
\begin{align}
d -d'\ \leq \,\ &\frac{(s -s')^2}{2(c-c')} + \frac{1}{24}(c-c')^3 \label{dd'}\\
\overset{\eqref{s-1}}{\leq}\ &\frac{(c-c')}{2} \left(-\frac{n_0}{4}\right)^{\!2}+ O(1)\=\frac{(c-c')}{32}n_0^2+O(1).\label{dd'-1}
\end{align}

Next we do the same for $\ch\_H(E_1)$ at the point
$$
b\,=\,0,\quad w\,=\,\frac{n^2}2-s+(n+c)\frac{s-s'}{c-c'}
$$
of the line $\ell$ \eqref{ell}. Since $\ell$ is above $\ell_f\cap U$ this lies in $U$. 
Arranged in powers of $n$ \eqref{boglinear} gives
\begin{align}\nonumber
\frac{n^3}{2}&c' +n^2\!\left(\frac{c'\;^2}{2} +2c'\frac{s-s'}{c-c'} -s'\right) +n \left(-2c's+\frac{s-s'}{c-c'}(2c'c+c'\;^2+2s')  \right)+2s'\;^2+\\
&(c'\;^2+2s')\left(\frac{sc'-cs'}{c-c'}\right)\ \geq\ 3d'(n+c') \ \stackrel{\eqref{dd'}}{\ge}\ 3(n+c')\left[d- \frac{(s-s')^2}{2(c-c')} - \frac{1}{24}(c-c')^3     \right]\!.\label{BGd'}
\end{align}
By \eqref{del}, \eqref{s'} and the bounds on $n,s,d$ in Definition \ref{close}, each term above is $\le O(n_0^3)$. Discarding $O(n_0^2)$ terms shows the following is nonnegative,
\begin{align}\nonumber
\frac{n_0^3}2c'+\,&n_0^2\left(-s'+2c'\frac{s-s'}{c-c'}\right)+2n_0s'\frac{s-s'}{c-c'}
-3n_0d+\frac{3n_0}2\frac{(s-s')^2}{c-c'}\\
&\=\frac{n_0^3}2c'+\frac{n_0}{c-c'}\Big[-n_0c's'-n_0cs'-ss'+2n_0c's-\tfrac12(s')^2+\tfrac32s^2-3dc+3dc'\Big].\label{long}
\end{align}
By $-s\le n_0\delta_n+|s_0|$ from Definition \ref{close}, the first three terms in the square bracket are
$$
-n_0c's'-n_0cs'-ss'\,\stackrel{\eqref{s'}}\le\,
n_0s'\(-c'-c+\delta_n+O\(\tfrac1{n_0}\)\)\,
\stackrel{\eqref{del}}{\le}\,\(\!-\tfrac23n_0c+O(1)\)s'\,\stackrel{\eqref{s'}}{\le}\,0,
$$
so removing them --- and the fifth bracketed term $-\frac12(s')^2\le0$ --- from \eqref{long} we deduce
\beq{ded}
\tfrac12c'n_0^3+\frac{n_0}{c-c'}\Big[2n_0sc'+\tfrac32s^2-3dc+3dc'\Big]\ \ge\ 0.
\eeq
Since $d\ge\frac13n_0^2\delta_n+ O(n_0)$ by Definition \ref{close}, the coefficient of $c'$ inside the bracket is
$$
2n_0s+3d\ \ge\ -2n_0^2\delta_n+n_0^2\delta_n+O(n_0)\=-n_0^2\delta_n+O(n_0).
$$
Using this and $0\ge c'\ge\delta_n-\frac13c$ \eqref{del} in the inequality \eqref{ded}, and discarding $O(n_0^2)$ terms,
\beqa
0 &\le& \tfrac12c'n_0^3+\frac{n_0}{c-c'}\Big[-n_0^2\delta_n\(\delta_n-\tfrac13c\)+\tfrac32n_0^2\delta_n^2-cn_0^2\delta_n \Big] \\
&=& \tfrac12c'n_0^3+\frac{n_0^3\delta_n}{c-c'}\Big[\tfrac12\delta_n-\tfrac23c\Big].
\eeqa
By Definition \ref{close}, $\delta_n\le\frac c3$ so the term in square brackets is $<0$. Thus both terms are $\le0$ and the only way the inequality can hold is if $c'=0=\delta_n$.
\end{proof}

\begin{Prop}\label{T}
	If $w_n=\vno$ and $c'=0$ then $E_1\cong T(-n_0)[1]$ for some $T\in\Pic\_0(X)$ \eqref{Pic0}.
\end{Prop}

\begin{proof}
	When $c'=0$ and $w_n = \vno$ the inequality \eqref{BGd'} becomes
	$$
	-s'n_0^2+\frac{2n_0s'(s_0-s')}{c}
 \ \ge\ 3n_0\left(d_0-\frac{(s_0-s')^2}{2c}-\frac1{24}c^3\right).
	$$
	Discarding terms of $O(n_0)$, this becomes 
	\begin{equation*}
	n_0s' \left(-n_0  - \frac{s'}{2c}  -\frac{s_0}c   \right)\ \ge\ O(n_0).
	\end{equation*}
Since $s'\ge0$ \eqref{s'} this forces $s'=0$, which is $2\Delta_H(E_1) = 0$. By \cite[Corollary 3.10]{BMS} or \cite[Lemma 3.2]{FT3} this implies there is no wall for $E_1$ in $U$. Thus $E_1$ is $\nubw$-semistable for any $(b,w) \in U$, so we may apply \cite[Lemma 2.6]{FT3} to conclude it is $T(-n_0)[1]$ for some $T\in\Pic\_0(X)$.
\end{proof}


The first case of the next result will allow us to do an induction on $\ch_1\!.\;H^2$, which decreases on passing from $E$ to $E_1$ since $\ch_1(E_1).H^2<\ch_1(w_n).H^2$ by \eqref{vec}. (We will deal with the second case afterwards.) Let $b_1<b_2$ be the $b$-values of the intersection points $\ell\cap\partial U$.

\begin{Prop}\label{prop.c' positive}
	If $c' > 0$, then either
	\begin{itemize}
	\item $(b,w)\in U([E_1])$ and $[E_1]$ is close to $\vno$, or
	\item $\ch_1(E_1).H^2 +b_1H^3<\min(c,\,b_2 -b_1)H^3$.
 This always holds when $c'<\delta_n+\frac23c$.
	\end{itemize}
\end{Prop}

\begin{proof} We consider two cases depending on the size of $c'$. \smallskip

\textbf{Case (i):} $c'<\delta_n+\frac23c$. Since $c', \delta_n, c \in \frac{1}{H^3}\mathbb{Z}$ this is $c'\le\delta_n+\frac23c-\frac1{3H^3}$, so
\beq{trayn}
n +c'\=n_0-\delta_n+c'\ \le\ n_0+\tfrac23c-\tfrac1{3H^3}.
\eeq
Since $(b,w)\in U(w_n)$ lies on or above $\ell_f$ by Definition \ref{close}, we have
\beq{bbbb}
b_1\ \leq\ b_1^f\=-n_0 + \tfrac13c + O\(\tfrac1{n_0}\) \quad\text{and}\quad
b_2-b_1\ \ge\ b_2^f-b_1^f\=\tfrac32n_0+O(1)
\eeq
by \eqref{b1b2}. Combining the first with \eqref{trayn} gives
\beq{lastn}
	\tfrac1{H^3}\ch_1(E_1).H^2 + b_1 \= n +c' +b_1\ \leq\ n_0 + \tfrac{2}{3}c - \tfrac1{3H^3} -n_0 + \tfrac13c + O\(\tfrac{1}{n_0}\)\ <\ c.
\eeq
Since \eqref{bbbb} shows $c<b_2-b_1$, this means $\ch_1(E_1).H^2 +b_1H^3<\min(c,\, b_2 -b_1)H^3$. \smallskip
	
\textbf{Case (ii):} $c'\ge\delta_n + \frac23c$. Since $\delta_n \geq 0$ and $c-c' >0$ this gives $c'\in\big[\frac{2}{3}c,c\)$. Let $\ell_{\js}(w_n)$ denote the \emph{Joyce-Song wall} for $w_n$ connecting
\beq{jswall}
\Pi(w_n)\=\(\!-n -c,-s+\tfrac12n ^2\)\ \text{ to }\ \Pi\(\cO_X(-n)\)\=\(\!-n ,\tfrac12n ^2\).
\eeq
This is easily calculated to intersect $\partial U$ at $b_1'=-n $ and $b_2'=n +2\frac{s}c$, so that $\ch_1(w_n).H^2 +b_1'H^3=cH^3=\min(c,b_2'-b_1')H^3$. 
If $\ell$ lies strictly above $\ell_{\js}$ it follows that
$$
\ch_1(w_n).H^2 +b_1H^3\ <\ \ch_1(w_n).H^2 +b_1'H^3\=\min(c,b'_2-b'_1)H^3\ \le\ \min(c,b_2-b_1)H^3,
$$
giving the second case of the Proposition. \medskip

So we now assume $\ell$ lies on or below $\ell_{\js}(w_n)$ and try to show that $[E_1]$ is close to $\vno$ by rewriting $\ch\_H(E_1)$ as
$$
\ch\_H(E_1)\=\(0,c,\tilde s,\tilde d\,\)-\ch\_H\!\(\cO_X(-n')\)\= \left(-1,\ n'+c,\ -\tfrac12{n'\;^2} + \tilde{s},\ \tfrac16{n'\;^3} +\tilde{d}\,\right),
$$
where $n' := n-c+c'\in\frac1{H^3}\Z$ and $\tilde{s}, \tilde{d} \in\Q$ are defined by
\begin{equation}\label{s}
-\tfrac12(n-c+c')^2 +\tilde{s} \= -\tfrac12n^2 +s' \quad\so\quad  \tilde{s} \= s'-(c-c')(n_0-\delta_n) + \tfrac12(c-c')^2,
\end{equation}
\begin{equation}\label{d}
\tfrac16(n-c+c')^3 +\tilde{d} \= \tfrac16n^3 +d' \quad \so \quad \tilde{d}\=d'+\tfrac12n^2(c-c') -\tfrac12n(c-c')^2 + \tfrac{1}{6}(c-c')^3.
\end{equation}
Setting $\delta_n' := n_0-n'=\delta_n + c-c'\in\frac1{H^3}\Z$, the inequality $c'\ge\delta_n + \frac23c$ gives $\delta_n'\in\(0,\frac c3\big]$. This is one of the bounds required by Definition \ref{close}; we need to show that $\tilde s,\tilde d$ satisfy the other two.

By \eqref{slope} the slopes of both $\ell_f$ and $\ell(w_n)$ of Figure \ref{figUwm} are $>-\frac14{n_0} -|s_0|H^3$. Since $\ell$ lies between $\ell(w_n)$ and $\ell_{\js}(w_n)$ \eqref{jswall}, and all three pass through $\Pi(w_n)$, their gradients are ordered
\begin{equation}\label{order}
    -\frac{n_0}{4} -|s_0|H^3\ <\ \frac{s -s'}{c-c'} \ \leq\ \frac{s}{c}\,,
\end{equation}
which implies 
\beq{ord2}
    \frac{c'}{c}s\ \le\ s'\ <\  \left(\tfrac14n_0 + |s_0|H^3\right)(c-c') +s.
\eeq
Together with \eqref{s} and $s\ge-n_0\delta_n  -|s_0|$ from Definition \ref{close}, this bounds $\tilde s$ below by
\beqa
    \tilde{s} &\ge& \frac{c'}{c}s - (c-c')(n_0-\delta_n) + \tfrac12(c-c')^2\\
    &>& \(\!-n_0\delta_n  -|s_0|\) -(c-c')n_0\\
    &=& -n_0\delta'_n -|s_0|. 
\eeqa
Similarly using \eqref{s}, \eqref{ord2} and  $s\le-\frac34n_0\delta_n+\(c+|s_0|H^3\)\delta_n+s_0$ from Definition \ref{close}, we bound $\tilde s$ above by
\beqa
   \tilde{s} &<& \left(\tfrac14n_0 + |s_0|H^3\right)(c-c')+s  - (c-c')(n_0-\delta_n) + \tfrac12(c-c')^2  \\
   &=& -\tfrac{3}{4}n_0(c-c') + (c-c')\left(|s_0|H^3 + \delta_n + \tfrac12(c-c')   \right) +s\\
   &<& -\tfrac{3}{4}n_0(c-c') + (c-c')\left(|s_0|H^3 + c   \right) -\tfrac{3}{4}n_0\delta_n + \delta_n(c+|s_0|H^3) +s_0\\
   &=& -\tfrac{3}{4}n_0\delta_n' + \delta_n'\left(|s_0|H^3 + c   \right) +s_0,
\eeqa
where in passing from the second line to the third we used that $\delta_n$ and $c-c'$ both lie in $\big[0,\frac c3\big]$, so $\delta_n+\frac12(c-c')<c$.
Thus $\tilde s$ satisfies Definition \ref{close}.\medskip

Now we turn to $\tilde d$. Combining \eqref{order} with the upper bound $s\le s_0$ of Definition \ref{close} gives $\(\frac{s-s'}{c-c'}\)^2\le\left(-\tfrac14n_0 -|s_0|H^3 \right)^2$. With \eqref{dd'} this gives
$$
    d' \ \ge\ d-\tfrac12(c-c')\left(\tfrac14n_0 +|s_0|H^3 \right)^2   - \tfrac{1}{24}(c-c')^3.
$$
Together with \eqref{d} this gives
$$
\tilde d\,\ge\,d-\tfrac12(c-c')\left(\tfrac14n_0 +|s_0|H^3 \right)^2\!-\tfrac{1}{24}(c-c')^3+\tfrac12(c-c')n^2 -\tfrac12(c-c')^2n + \tfrac{1}{6}(c-c')^3.
$$
Substituting $n=n_0-\delta_n$ and recalling that $\delta_n+\frac12(c-c')<c$,
\beqa
    \tilde{d}
    &\ge& d -\tfrac12(c-c')\left(\tfrac14n_0 +|s_0|H^3 \right)^2  +\tfrac12(c-c')n_0^2 -\delta_n(c-c')n_0 - \tfrac12(c-c')^2n_0 \\
    &>& d +\tfrac{15}{32}(c-c')n_0^2 -n_0(c-c') \left( \tfrac14|s_0|H^3 +\delta_n + \tfrac12(c-c')          \right) - \tfrac12(c-c')(s_0H^3)^2 \\
    &>& d +\tfrac{15}{32}(c-c')n_0^2 -n_0(c-c') \left( |s_0|H^3 +c\right) - (c-c')(s_0H^3)^2.
\eeqa
Comparing with $d\,\ge\,\frac{15}{32}n_0^2\delta_n -n_0\delta_n\(|s_0|H^3 +c\) -\delta_n\(s_0H^3\)^2 + d_0$ from Definition \ref{close} and using $\delta_n'=\delta_n+c-c'$ this finally gives
$$
\tilde d\ >\ \frac{15}{32}n_0^2\delta'_n -n_0\delta'_n\(|s_0|H^3 +c\) -\delta'_n\(s_0H^3\)^2 + d_0.
$$
Thus $\tilde d$ satisfies Definition \ref{close}.
Finally since the slope of $\ell$ is $\ge$ that of $\ell_f$ it is also $\ge$ that of $\ell([E_1])$. So $(b,w)$ lies on or above $\ell([E_1])$ and therefore in $U([E_1])$ as claimed.
\end{proof}

We can phrase the second possibility of Proposition \ref{prop.c' positive} by saying that $(b,w)$ lies in $E_1$'s \emph{safe area $U^s_{[E_1]}\subset U$} defined as follows (see also Figure \ref{safef}).
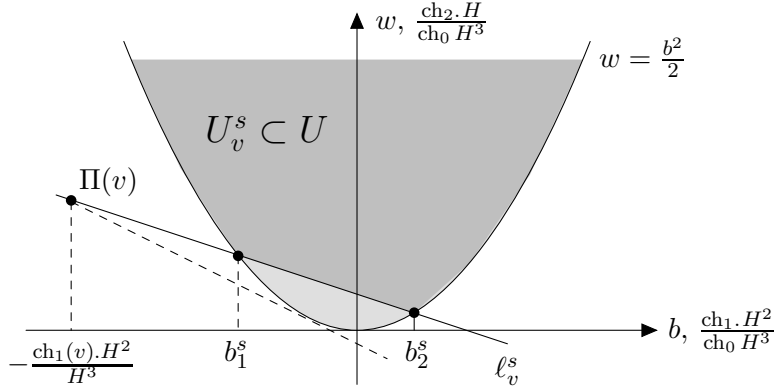
\begin{figure}[h]
	\begin{centering}
		\definecolor{zzttqq}{rgb}{0.27,0.27,0.27}
		\definecolor{qqqqff}{rgb}{0.33,0.33,0.33}
		\definecolor{uququq}{rgb}{0.25,0.25,0.25}
		\definecolor{xdxdff}{rgb}{0.66,0.66,0.66}
		
		\begin{tikzpicture}[line cap=round,line join=round,>=triangle 45,x=1.0cm,y=0.9cm]
		
		\draw[->,color=black] (-4.4,0) -- (4,0);
		\draw  (4, 0) node [right ] {$b,\,\frac{\ch_1\!.\;H^2}{\ch_0H^3}$};

		\fill [fill=gray!25!white] (0,0) parabola (3,4) parabola [bend at end] (-3,4) parabola [bend at end] (0,0);
		\fill[fill=gray!50!white](-3, 4)--(-2.35,2.4)--(-1.58,1.1)--(.76,.26)--(1.5,1)--(2.3,2.3)--(3,4); 

		\draw(0,0) parabola (3.1,4.27); 
		\draw(0,0) parabola (-3.1,4.27); 
		\draw  (3.8 , 3.6) node [above] {$w= \frac{b^2}{2}$};
		
		\draw[->,color=black] (0,-.8) -- (0,4.7);
		\draw[dashed,color=black](-3.8,1.9)--(.4,-0.42);
		\draw  (1, 4.1) node [above ] {$w,\,\frac{\ch_2\!.\;H}{\ch_0H^3}$};
		
		\draw [color=black] (-4, 2) -- (2,-.2);
		\draw[color=black](.76,.26)--(.76,0);
		\draw[dashed,color=black](-1.58,1.1)--(-1.58,0);
		\draw[dashed,color=black](-3.8,1.9)--(-3.8,0);
		\draw  (-3.3, 2.2) node {$\Pi(v)$};
		\draw  (-1.2, 2.5) node [above] {\Large{$U^s_v\subset U$}};
		\draw  (.8, 0) node [below] {$b_2^s$};
		\draw  (-1.6, 0) node [below] {$b_1^s$};
		\draw  (-3.8, 0) node [below] {$-\frac{\ch_1(v).H^2}{H^3}$};
		\draw  (2, -.6) node {$\ell^s_v$};
		\begin{scriptsize}
		\fill (-3.8,1.92) circle (2pt);
		\fill (.76,.26) circle (2pt);
		\fill (-1.58,1.1) circle (2pt);
		
		\end{scriptsize}
		
		\end{tikzpicture}
		
		\caption{The safe area $U^s_v\subset U$ of a rank $-1$ class $v$ with $\Delta_H(v)>0$}
		
		\label{safef}
		
	\end{centering}
\end{figure}

Let $v\in K(X)$ be any class of rank $-1$ and $\Delta_H(v)\ge0$ . There is a unique line passing through $\Pi(v)$ which is tangent to $\partial U=\big\{w=\frac12b^2\big\}$ at a point to the right of $\Pi(v)$ (or at $\Pi(v)$ if $\Delta_H(v)=0$). Rotating the line anticlockwise about $\Pi(v)$, it intersects $\partial U$ in two points with $b$-values $b_1^s\le b_2^s$ such that $b_2^s-b_1^s$ moves monotonically through $[0,\infty)$, while the horizontal distance $b_1^s+\ch_1(v).H^2/H^3$ from $b_1^s$ to $\Pi(v)$ decreases to 0. Thus there is a unique line $\ell^s_v$ passing through $\Pi(v)$ such that 
\begin{equation}\label{safe}
\ch_1(v).H^2 +b_1^sH^3 \= \min(c,\,b_2^s -b_1^s)H^3.  
\end{equation}
For lines through $\Pi(v)$ strictly above $\ell^s_v$, with $b$-values $b_1'<b_2'$ at the two intersections points with  $\partial U$, the equality is replaced by the inequality
\beq{safe<}
\ch_1(v).H^2 +b_1'H^3\ <\ \min(cH^3,\,b_2' -b_1')H^3,
\eeq
while below we get the opposite inequality $>$. A similar definition was used in \cite[Section 3]{FT3} for rank $\ge0$ classes. Here we include the extra $cH^3$ term in order to work above the Joyce-Song wall $\ell_{\js}(w_n)$ for $w_n$, which we observed in \eqref{jswall} satisfies \eqref{safe}, so that
$$
\ell_{\js}(w_n)\=\ell^s_{w_n }.
$$

\begin{Def}\label{defsafe}
For a class $v \in K(X)$ of rank $-1$ with $\Delta_H(v) \geq 0$ the safe area $U^s_v \subset U$ is the set of $(b,w)$ which lie strictly above $\ell^s_v$ and to the right of $\Pi(v)$, i.e. $bH^3 > -\ch_1(v).H^2$. 
\end{Def}

If we manage to wall cross into the safe area, we stay there, in the following sense.

\begin{Lem}\label{lem.safe area}
Suppose $\rk(v)=-1,\ (b,w)\in U^s_v$ and $E$ is $\nubw$-semistable of class $v$. Then any $\nubw$-destabilising sequence 
$E' \hookrightarrow E \twoheadrightarrow E''$ satisfies 
	\begin{itemize}
		\item one of the factors $E_0$ is a rank zero sheaf with $\ch_1(E_0).H^2\in(0, cH^3)$, and 
		\item the other factor $E_1$ has rank $-1$ with $(b,w)$ inside its safe area $U_{[E_1]}^s$. 
	\end{itemize}
\end{Lem}

\begin{proof}
	The proof is similar to that of Lemma \ref{lem.destabilising objects}. One of the objects $E',E''$ has rank$\,<0$; call it $E_1$. The other $E_0$ has rank$\,\ge0$. Let $r_i:=\rk\!\(\cH^{-1}(E_i)\)\ge0$. Then since $E_i\in\Ab$ for $b\in[b_1^s,b_2^s]$, 
	\beqa
	\ch_1\!\(\cH^{-1}(E_i)\).H^2 &\le& b_1^sr_iH^3 \\
	\text{and}\quad
	\ch_1\!\(\cH^0(E_i)\).H^2 &\ge& b_2^s(\ch_0(E_i)+r_i)H^3,
	\eeqa
	with the first inequality strict for $i=1$. Subtracting gives the inequality
	\begin{equation}\label{strict}
	    	\ch_1(E_i).H^2\ \geq\ b_2^s\ch_0(E_i)H^3+(b_2^s -b_1^s)r_iH^3,
	\end{equation}
	which is again strict when $i=1$. Adding over $i=0,1$ implies
	\beq{name}
	\ch_1(E).H^2 - b_2^s\ch_0(E)H^3 
	\ >\ (b_2^s -b_1^s)(r_0+r_1)H^3.
	\eeq
Since $\ch_0(E)=-1$ the left hand side is
$$
\ch_1(E).H^2 +b_1^sH^3+(b_2^s-b_1^s)H^3\ \stackrel{\eqref{safe<}}{<}\ 2(b_2^s -b_1^s)H^3,
$$
so by \eqref{name} this gives $r_0+r_1<2$. Therefore we can repeat the argument \eqref{LES} to show that $E_0$ is a sheaf of rank zero and  $E_1$ is a complex of rank $-1$ with rank zero $\cH^{0}(E_1)$.
	
	Thus $\ch_1(E_0).H^2\ge0$, and it cannot be zero since $\nubw(E_0)=\nubw(E)<+\infty$. So
	\beq{ineq}
	\ch_1(E_0).H^2\ >\ 0\quad\text{and}\quad \ch_1(E_1).H^2\ <\ \ch_1(E).H^2.
	\eeq
	The line through $\Pi(E)$ and $(b,w)$ passes through the safe area $U^s_{[E]}$ so satisfies \eqref{safe<}. Hence by the second inequality of \eqref{ineq} it also satisfies \eqref{safe<} for $v=[E_1]$. Since it also goes through $\Pi(E_1)$, we conclude that it lies in the safe area $U^s_{[E_1]}$.
	
	Finally $\ch_1^{b_1^sH}(E).H^2\le cH^3$ by \eqref{safe} and $\ch_1^{b_1^sH}(E_1).H^2> 0$ by \eqref{strict}, so
	\[
	\ch_1(E_0).H^2 \= \ch_1^{b_1^sH}(E_0).H^2 \= \ch_1^{b_1^sH}(E).H^2 - \ch_1^{b_1^sH}(E_1).H^2\ <\ cH^3. \qedhere
	\]
\end{proof}

\subsection{Summary}\label{summary} We summarise what we have proved so far. (1) and (2) are the content of Lemma \ref{lem.destabilising objects}, (3) is Proposition \ref{T}, (4) is Proposition \ref{c' negative} and (5) is Proposition \ref{prop.c' positive}.

\begin{enumerate}
    \item[(1)] Any $\nubw$-semistable object $E$ of class $\vno$ is a 2-term complex with $\rk\!\(\cH^{-1}(E)\)=1$ and $\rk\!\(\cH^0(E)\)=0$.
\item[(2)] If $E$ is destabilised on a wall $\ell$ the possible semistable factors are a rank $-1$ complex $E_1$ with $\ch_1(E_1).H^2<\ch_1(E).H^2$, and rank 0 sheaves $E_0$.
\item[(3)] If $\ch_1(E_0).H^2=cH^3$ then our wall is $\ell_{\js}(\vno)$ \eqref{jswall} and $E_1\cong T(-n_0)[1]$.
\item[(4)] Otherwise $\ch_1(E_0).H^2\in(0,cH^3)$.
\item[(5)] As we increase $w$, moving above $\ell$, the above results (1),\,(2),\,(4) for $E$ also hold for $E_1$ (and its rank $-1$ destabilising factors, etc). This is because either
\begin{enumerate}
\item[(a)] $[E_1]$ is close to $\vno$ in the sense of Definition \ref{close} and $(b,w)\in U([E_1])$ by Figure \ref{pic}, so the results for $E$ apply to $E_1$ immediately, or
\item[(b)] $(b,w)$ is in the safe area for $E_1$, where the results follow from Lemma \ref{lem.safe area}.
\end{enumerate}
\end{enumerate}
\begin{figure}[h]
	\begin{centering}
		\definecolor{zzttqq}{rgb}{0.27,0.27,0.27}
		\definecolor{qqqqff}{rgb}{0.33,0.33,0.33}
		\definecolor{uququq}{rgb}{0.25,0.25,0.25}
		\definecolor{xdxdff}{rgb}{0.66,0.66,0.66}
		
		\begin{tikzpicture}[line cap=round,line join=round,>=triangle 45,x=1.0cm,y=1.0cm]
		
		\draw[->,color=black] (-4,0) -- (4,0);
		\draw  (4, 0) node [right ] {$b,\,\frac{\ch_1\!.\;H^2}{\ch_0H^3}$};
		
		\fill [fill=gray!25!white] (0,0) parabola (3,4) parabola [bend at end] (-3,4) parabola [bend at end] (0,0);
		
	    \fill[fill=gray!50!white](-2.6, 4)--(-2.6,3)--(-2.25,2.3)--(1.65,1.25)--(2.4,2.5)--(3,4); 
	   
		\draw  (0,0) parabola (3.1,4.27); 
		\draw  (0,0) parabola (-3.1,4.27); 
		\draw  (3.8 , 3.6) node [above] {$w= \frac{b^2}{2}$};
	
		\draw[->,color=black] (0,-.8) -- (0,4.7);
		\draw  (1, 4.3) node [above ] {$w,\,\frac{\ch_2\!.\;H}{\ch_0H^3}$};
		
		\draw [color=black] (-3.4, 1.95) -- (0.6, 4.18);
		\draw [color=black] (-3.5, 2.136) -- (2.5, .5);
		\draw [color=black] (-2.8, 2.45) -- (2.5, 1);
	    \draw [dashed] (-2.6,2.4) -- (-2.6,3);
	    \draw [thick] (-2.6,3) -- (-2.6,4.2);
	    
	    \draw  (-3.2, 1.7) node {{\small$\Pi(\vno)$}};
	    \draw  (-3.2, 2.7) node {{\small$\Pi(E_1)$}};
		\draw  (2, .35) node {$\ell_f$};
		\draw  (2.7, 1.3) node {$\ell([E_1])$};
        \draw  (1, 2.5) node [above] {\Large{$U([E_1])$}};
        \draw  (-1.2, 3.5) node [left] {$(b,w)$};
		\fill (-2.6,2.39) circle (2pt);
		\fill (-3.2,2.06) circle (2pt);
		\fill (-1.2,3.5) circle (2pt);

		\end{tikzpicture}
		
		\caption{Diagram showing why $(b,w)\in U([E_1])$ in 5(a) of the Summary}
		\label{pic}
	\end{centering}
\end{figure}
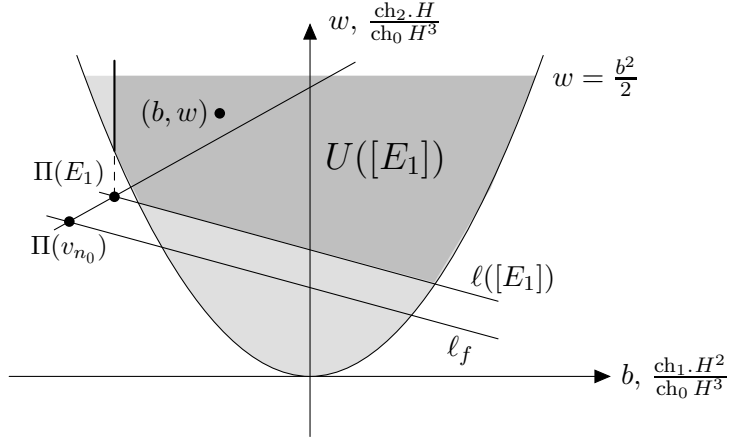

In particular $\ch_1\!.H^2$ drops any time we cross a wall and pass from a rank $-1$ complex to a rank $-1$ semistable factor (and then cross further walls for that, and pass to its rank $-1$ semistable factors, etc). This process terminates in finitely many steps because $\ch_1\!.\;H^2$ cannot drop below $n_0-\frac c3$ by Lemma \ref{lem.destabilising objects}. \medskip

%
%
%

We note in passing that, with a little further effort, the results of this Section can be used to prove
that on crossing a wall, those semistable objects which become unstable have Harder-Narasimhan filtrations of length 2 only. The semistable factors have ranks 0 and $-1$. However the rank 0 factor $E_0$ can be strictly semistable, with further decompositions into semistable factors, all with $\ch_H^{\le2}$ proportional to $\ch_H^{\le2}(E_0)$. All of these factors enter into the wall crossing formula, which therefore does not simplify much. So we omit the details of this result and do not use it.

\subsection{The quintic threefold} \label{Chunyi}
		Let $X\subset\PP^4$ be a smooth quintic 3-fold. Chunyi Li \cite[Theoreom 2.8]{Li} has proved the Bogomolov-Gieseker inequality \eqref{quadratic form} for a restricted subset of weak stability conditions on $X$: those $(b,w)$ satisfying 
	\begin{equation}\label{in for b, w}
	w\ >\ \tfrac{1}{2} b^2 + \tfrac{1}{2}\big(b - \lfloor b \rfloor\big)\big (\lfloor b \rfloor+1 - b \big). 
	\end{equation}
	In particular Conjecture \ref{conjecture} holds for any $w > \frac12b^2$ when $b \in \mathbb{Z}$. 

We check that knowing this restricted form of the Bogomolov-Gieseker inequality \eqref{quadratic form} is sufficient for our purposes. In this Section we have applied \eqref{quadratic form} to
	\begin{enumerate}
		 \item objects $E$ of class $\vno$ to find $\ell_f$ \eqref{lfeq}, and 
		\item the rank $-1$ destabilising objects $E_1$ of an object of class $w_n$ close to $\vno$ in \eqref{BGd'}. 
\end{enumerate}	
Now $(b,w)\in\ell_f\cap U$ in case (a), while in (b) we worked at a point $(b,w)\in U(w_m)$ which is therefore above $\ell_f\cap U$. But by \eqref{b1b2} we know $b$ takes an integer value along $\ell_f\cap U$, so by \eqref{in for b, w} we can apply \eqref{quadratic form} at $b\in\Z$ to prove the same results in cases (a) and (b).

The only other place we applied \eqref{quadratic form} is in \eqref{dd'}, where we invoked \cite[Lemma B.3]{FT2} for $E_0$. Setting $v=[E_0]$, the proof of that Lemma applied \eqref{quadratic form} along a line $\ell_v$ which intersects $\partial U$ at points $a_v < b_v$ with $b_v -a_v = \ch_1(v).H^2/H^3$. This is $\geq 1$ because $\Pic(X)=\Z.H$, so $b$ takes an integer value on the closure of $\ell_v\cap U$, so again \eqref{in for b, w} holds at a point of $\ell_v\cap U$.


%
%

\section{Wall crossing}\label{wcross}
We now assume the Calabi-Yau condition $K_X\cong\cO_X$. If $h^1(\cO_X)>0$ the Jac$\;(X)$ action on moduli spaces of sheaves forces $\J(v)=0$ whenever $\rk(v)>0$, so Theorem \ref{1} already holds.
So in this Section we further assume that $H^1(\cO_X)=0$ and prove Theorem \ref{2} for the class $\mathsf v$ \eqref{vdef} of rank 0 and dimension 2, by applying wall crossing to the class $\vno$ of \eqref{class vn}.

\subsection*{Wall crossing formula}
The work of Joyce-Song \cite{JS} can be used to define generalised DT invariants counting $\nubw$-semistable objects of class $v\in K(X)$ with $\nubw(v)<+\infty$,
\beq{DTJ}
\J_{b,w}(v)\ \in\ \Q.
\eeq
In \cite[Section 4 and Appendix C]{FT3} we described these invariants and showed they are well defined so long as $\nubw(v)<+\infty$. When $w\gg0$, so that $(b,w)$ lies in $v$'s large volume chamber of Proposition \ref{large}, we denote \eqref{DTJ} by $\J_{b,\infty}(v)$.

In \cite[Section 4]{FT3} we also showed the Joyce-Song wall crossing formula applies to the $\J_{b,w}(v)$ under the same $\nubw(v)<+\infty$ condition.

Suppose $\ell$ is the line through $\Pi(v)$ and a point $(b,w_0)\in U$ such that $\nu_{b, w_0}(v) < +\infty$.\footnote{If $\rk(v)=0$ and $\ch_1(v).H^2>0$ then $\ell$ is the line through $(b,w_0)$ of gradient $\ch_2(v).H\big/\!\ch_1(v).H^2$.} By the local finiteness of walls of Proposition \ref{prop. locally finite set of walls} we may choose
\beq{just}
(b,w_\pm)\,\in\,U\,\text{ just above and below the wall }\ell,
\eeq
in the sense that $(b,w_\pm)\not\in\ell$ and between $(b,w_-)$ and $(b,w_+)$ there are no walls for $v$, \emph{nor any of its finitely many semistable factors}, except for $\ell$.
Then the wall crossing formula is
\beq{WCF}
\quad\J_{b,w_-}(v)\=\J_{b,w_+}(v)\ +\hspace{-2mm}\mathop{\sum_{m\ge2,\ \alpha_1,\dots,\;\alpha_m\,\in\,C(\Ab),}}_{\sum_{i=1}^m\alpha_i\,=\,v,\ \nu_{b,w_0}(\alpha_i)\,=\,\nu_{b,w_0}(v)\,\forall i}\hspace{-3mm}C_{+,-}(\alpha_1,\dots,\alpha_m)\prod_{i=1}^m\J_{b,w_+}(\alpha_i).
\eeq
Here $C(\Ab)$ is the positive cone \eqref{CAb}, and the $C_{+,-}(\alpha_1,\dots,\alpha_n)\in\Q$ are universal coefficients  depending only on the Mukai pairings $\chi(\alpha_i,\alpha_j)$ and the relative sizes of the set of slopes $\{\nu\_{b,w_\pm}(\alpha_i)\}$. The term
$$
C_{+,-}(\alpha_1,\dots,\alpha_m)\prod_{i=1}^m\J_{b,w_+}(\alpha_i)
$$
is zero unless there is a $\nu\_{b,w_0}$-semistable object $E$ of class $v$ with $\nu\_{b,w_0}$-semistable factors of classes $\alpha_1,\dots,\alpha_m$. (Even most of these terms vanish --- although \eqref{WCF} is a countable sum, only finitely many terms in it are nonzero.) The formula reflects the different Harder-Narasimhan filtrations of $E$ on the two sides of the wall, and then further filtrations of the semistable Harder-Narasimhan factors by semi-destabilising subobjects. A similar formula holds if we swap $(b,w_-)$ with $(b,w_+)$.

The coefficients $C_{+,-}$ are given by complicated formulae. The only explicit expression we will need is that when $m =2$, $\nu\_{b,w_0}(\alpha_1)=\nu\_{b,w_0}(\alpha_2),\ \nu\_{b,w_+}(\alpha_1)>\nu\_{b,w_+}(\alpha_2)$ and $\nu\_{b,w_-}(\alpha_1)<\nu\_{b,w_-}(\alpha_2)$ then
\beq{><}
C_{+,-}(\alpha_1,\alpha_2)+C_{+,-}(\alpha_2,\alpha_1)\=(-1)^{\chi(\alpha_1,\alpha_2)-1}\chi(\alpha_1,\alpha_2).
\eeq

For economy of notation we incorporate the sign into the Mukai pairing by setting
$$
\bar\chi(v,w)\ :=\ (-1)^{\chi(v,w)-1}\chi(v,w), \qquad \bar\chi(w(d))\ :=\ (-1)^{\chi(w(d))-1}\chi(w(d)),
$$
for $v,w\in K(X)$. The second is the special case of the first where $v=[\cO_X(-d)]$.

\subsection*{The base case} We are now ready to apply the wall crossing formula to the class $\vno$ \eqref{class vn}. We will induct on the value of $c$, so we begin by fixing any class $\v$ \eqref{vdef} with
$$
\tfrac1{H^3}\ch_1(\v).H^2\=c_{\min}\ \in\ \tfrac1{H^3}\N,
$$
where $c_{\min}$ is the minimal value of $D.H^2/H^3$ over effective divisors $D\ne0$. 

We choose $n_0\gg0$ and always work with $(b,w)\in U$ to the right of $\Pi(\vno)$ --- i.e. with $b>-(c_{\min}+n_0)$ --- to ensure that $\nubw(\vno)<+\infty$. For $(b,w)$ below $\ell_f$ \eqref{lfeq} all objects of class $\vno$ are unstable. For $(b,w)$ on or above $\ell_f\cap U$ we have $(b,w)\in U(\vno)$, where we have the analysis of possible destabilising sequences in Summary \ref{summary}. The rank 0 factor $E_0$ has $\ch_1(E_0).H^2=(c-c')H^3$, where $c'\in[0,c_{\min})$. Thus $c'=0$ and the other factor $E_1$ is $T(-n_0)[1]$ for some $T\in\Pic\_0(X)$.

Thus we are on the Joyce-Song wall $\ell_{\js}(\vno)$, and this is the \emph{only wall} in $U$ for $\vno$. Let $(b,w_\pm)$ be points just above and below $\ell_{\js}(\vno)$. Since $c$ is minimal $E_0$ is \emph{stable} with no semistable factors. So using \eqref{><} the wall crossing formula simplifies to
$$
\J_{b,w_-}(\vno)\=\J_{b,w_+}(\vno)-\bar\chi(v(n_0))\cdot\J_{b,w_+}\big[T(-n)[1]\big]\cdot\J_{b,w_+}(\v).
$$
Here $\J_{b,w_-}(\vno)=0$ and $\J_{b,w_+}\big[T(-n)[1]\big]=\#H^2(X,\Z)_{\mathrm{tors}}$ is the invariant counting line bundles $T\in\Pic\_0(X)$.
There are no further walls for $\vno$, and there are no walls at all in $U(v_n)$ for $\v$ by Lemma \ref{rk0ss}, so the formula simplifies further to give
\beq{JJJ}
\J_{\infty}(\vno)\=\bar\chi(\v(n_0))\cdot\#H^2(X,\Z)_{\mathrm{tors}}\cdot\J_{b,\infty}(\v).
\eeq
As before $\J_{b,\infty}$ denotes $\J_{b,w}$ for $w\gg0$, while from now on $\J_\infty(\alpha)$, for a rank $-1$ class $\alpha$, will always denote $\J_{b,\infty}(\alpha)$ for $b>\mu\_H(\alpha)$ --- i.e. for $b$ to the right of $\Pi(\alpha)$. Rearranging \eqref{JJJ} proves a special case of the following.

\begin{Thm}\label{Ab} For any rank $0$ class $\v\in K(X)$ with $\ch_1(\v).H^2>0$ there is a universal formula $\J_{b,\infty}(\v)=F\(\J_{\infty}(\alpha_1),\J_{\infty}(\alpha_2),\dots\)$ with all $\alpha_i$ of rank $-1$.
\end{Thm}

Here $F$ is a polynomial with countably many terms, but only finitely many of them nonzero. It may change from line to line --- we only care that its coefficients depend on just the cohomology and Chern classes of $X$.

\subsection*{The induction step}
We will prove Theorem \ref{Ab} by an ascending induction on $c=\ch_1(\v).H^2/H^3$. The universal formula \eqref{JJJ} proves the base case $c=c_{\min}$. So we can inductively assume that we have proved Theorem \ref{Ab} for all rank 0 classes with $\ch_1\!.\;H^2<cH^3$ and next prove it for a fixed class $\v$ with $\ch_1(\v).H^2=cH^3$.

So we choose $n_0\gg0$ as in \eqref{nzero}, set $\vno=\v-[\cO_X(-n_0)]$ and take $(b,w)$ to the right of $\Pi(\vno)$. As we wall cross for the class $\vno$ we will produce rank $-1$ classes of smaller $\ch_1\!.\;H^2$, which we can also express in terms of invariants $\J_{\infty}(\beta_i)$ with $\rk(\beta_i)=-1$ by a separate induction, as follows.

\begin{Prop}\label{prop-new}
    	Fix a class $\beta$ of rank $-1$ with $\ch_1(\beta).H^2/H^3 \in \big[n_0-\frac c3,\,n_0+c\big)$ and a point $(b,w) \in U$ to the right of $\Pi(\beta)$, on or above $\ell_f$, not\;\footnote{This condition is not strictly necessary; on changing the universal constants $C_{+,-}$ the wall ``crossing" formula \eqref{WCF} still holds when one of $(b,w_\pm)$ lies \emph{on} the wall. Since we only care about crossing walls to reach the large volume chamber we do not concern ourselves with the values of the invariants $\J_{b,w}$ on walls.} on a wall of instability for $\beta$. Suppose that either 
    	\begin{enumerate}
    	    \item $(b,w)\in U^s_\beta$ is in the safe area of $\beta$, or 
    	    \item $\beta$ is close to $\vno$ and $(b,w) \in U([\beta])$.
    	\end{enumerate}
    	Then there is a universal formula $\J_{b,w}(\beta)=F\(\J_{\infty}(\beta_1),\J_{\infty}(\beta_2),\dots\)$ with $\rk(\beta_i)=-1$. 
\end{Prop}

\begin{proof}
We prove this by an ascending induction on $\ch_1(\beta).H^2/H^3$. If it takes the minimal value $\ge n_0-\frac c3$ attained by integral classes then there is no wall for $\beta$ on or above $\ell_f$ by Lemma \ref{lem.destabilising objects}, so the claim is true and the induction starts.

The claim is trivial when $(b,w)$ lies in the large volume chamber for $\beta$ of Proposition \ref{large}. We take this as the base case of a descending induction on the walls $\ell$ for $\beta$ above $\ell_f$. Since there are only finitely many of them, we may assume the claim holds at a point $(b,w_+)$ just above $\ell$, in the sense of \eqref{just}. Then we need only deduce from this, and the wall crossing formula \eqref{WCF}, that it also holds at a point $(b,w_-)$ just below $\ell$.

We analyse the formula \eqref{WCF} for $v=\beta$. For each term of the sum, 
Proposition \ref{prop.c' positive} and Lemma \ref{lem.safe area} show that precisely one $\alpha_i$ has rank $-1$; it then satisfies either (a) or (b) above with $\ch_1(\alpha_i).H^2 < \ch_1(\beta).H^2$, so by our induction on $\ch_1(\beta).H^2/H^3$ the invariant $\J_{b,w_+}(\alpha_i)$ can be written as universal expressions in invariants $\J_{\infty}$ of rank $-1$ classes.

The other terms $\alpha_j$ in \eqref{WCF} are all of rank zero with $\ch_1(\alpha_j).H^2 < cH^3$. They satisfy $\J_{b,w_+}(\alpha_i)=\J_{b,\infty}(\alpha_i)$ by Lemma \ref{rk0ss}. Thus our induction on $c$ writes these as universal expressions in invariants $\J_{\infty}$ of rank $-1$ classes too.
\end{proof}

Starting with $(b,w)$ just below $\ell_f$ --- the lower boundary of $U(\vno)$ --- all objects of class $\vno$ are strictly unstable and we have the universal formula
\beq{univ}
\J_{b,w}\(\vno\)\=0.
\eeq
Now we move upwards towards the large volume chamber $w\gg0$ for $\vno$, crossing finitely many walls $\ell\cap U$ for $\vno$ as we go and modifying the formula accordingly.

Let $(b,w_-),\,(b,w_0),\,(b,w_+)$ denote points just below, on and above $\ell$ in the sense of \eqref{just}. By induction on the walls we may assume that we have a universal formula for $\J_{b,w_-}(\vno)$ in terms of invariants $\J_{\infty}$ of rank $-1$ classes, plus the extra term 
\beq{JJJ2}
\bar\chi(\v(n_0))\cdot\#H^2(X,\Z)_{\mathrm{tors}}\cdot\J_{b,\infty}(\v)
\eeq
from \eqref{JJJ} \emph{if and only if $(b,w_-)$ is above $\ell_{\js}(\vno)$}. The induction step is to use the wall crossing formula across $\ell$ --- with $(b, w_+)$ and $(b, w_-)$ swapped in \eqref{WCF} --- to deduce a similar formula for $\J_{b,w_+}(\vno)$.

Since $\ell$ is on or above $\ell_f$ we have $(b,w_0)\in U(\vno)$, so Summary \ref{summary} applies. We conclude that in each term of the sum on the right hand side of \eqref{WCF} precisely one $\alpha_i$ is of rank $-1$ and satisfies either 
\begin{enumerate}
    \item [(i)] condition (a) or (b) of Proposition \ref{prop-new}, or 
    \item [(ii)] $[\alpha_i] = \big[\cO_X(-n)[1]\big]$.
\end{enumerate}
The other $\alpha_j$ are all of rank zero with $\ch_1(\alpha_j).H^2 \le cH^3$. Equality can only hold in case (ii); then $\ell=\ell_{\js}$, the corresponding term in \eqref{WCF} has $m=2$ and the terms $\alpha_1,\alpha_2$ are $\cO_X(-n)$ and $\v$ in some order, giving the contribution
$$
\bar\chi(\v(n_0))\cdot\#H^2(X,\Z)_{\mathrm{tors}}\cdot\J_{b,\infty}(\v)
$$
of \eqref{JJJ2}. Thus every other term $\J_{b,w_-}$ in the sum can be written in terms of invariants $\J_{\infty}$ of rank $-1$ classes by Proposition \ref{prop-new} and the induction assumption. Therefore our universal formula for $\J_{b,w_-}(\vno)$ induces one for $\J_{b,w_+}(\vno)$, with the extra term \eqref{JJJ2} if $(b,w_0)$ is on $\ell_{\js}$. This completes the induction on the walls, allowing us to move to the large volume chamber for $\vno$, where the universal formula now takes the form
$$
\J_{\infty}(\vno)\=F\(\J_{\infty}(\alpha_1),\J_{\infty}(\alpha_1),\dots\)+\bar\chi(\v(n_0))\cdot\#H^2(X,\Z)_{\mathrm{tors}}\cdot\J_{b,\infty}(\v),
$$
for some rank $-1$ classes $\alpha_i$. Rearranging proves Theorem \ref{Ab} for our class $\v$ and therefore, by induction, all $\v$.

\subsection*{Proof of Theorem \ref{1}}
As explained in the Introduction, to prove Theorem \ref{1} it is sufficient to prove Theorem \ref{2} for rank 0 classes $\v\in K(X)$ with $\ch_1(\v).H^2>0$. For these we now have Theorem \ref{Ab}.

By results of Toda \cite{TodaBG} explained in Appendix \ref{appendix}, the invariants $\J_{\infty}(\alpha_i)$ in Theorem \ref{Ab} are stable pair invariants up to multiplying by the factor $\#H^2(X,\Z)_{\mathrm{tors}}$. In turn we can replace these by DT invariants $\J(\beta_j)$ counting ideal sheaves\footnote{These ideal sheaves are all Gieseker stable, and --- after tensoring by elements of $\Pic\_0(X)$ --- account for all Gieseker semistable sheaves in the same class.} in rank 1 classes $\beta_j$, by the wall crossing formula \cite[Theorem 1.1]{BrDTPT}. This proves a version of Theorem \ref{1} for $\v$, but with $\J(\v)$ replaced by $\J_{b,\infty}(\v)$ on the left hand side,
\beq{nearly}
\J_{b,\infty}(\v)\=F\(\J(\beta_1),\J(\beta_2),\dots\).
\eeq

Finally we wall cross from large volume (or ``tilt") stability to Gieseker stability for $\mathsf v$. This was already carried out in \cite[Sections 4 and 5]{FT3}; we briefly review the details.

For any $\alpha\in K(X)$ we denote its Hilbert polynomial by
$$
P_\alpha(t)\ :=\ \chi(\alpha(t))\=a_dt^d+a_{d-1}t^{d-1}+\dots+a_0,
$$
for $d\le3$ such that $a_d\ne0$.
From this we define the reduced Hilbert polynomial $p_\alpha(t)$ and its truncation $\widetilde p_\alpha(t)$ by
$$
p_\alpha(t)\ :=\ \tfrac1{a_d}P(t) \quad\text{and}\quad \widetilde p_\alpha(t)\ :=\ p_\alpha(t)-\tfrac{a_0}{a_d}\=t^d+\dots+\tfrac{a_1}{a_d}t.
$$
There is a total order $\prec$ on $\{$monic polynomials$\}\sqcup\{0\}$ defined by $p\prec q$ if and only if
\begin{itemize}
\item[(i)] deg$\,p>\ $deg$\,q$, or
\item[(ii)] deg$\,p=\,\;$deg$\,q$ and $p(t)<q(t)$ for $t\gg0$.
\end{itemize}
(The strange direction of the inequality (i) ensures that sheaves of lower dimension have \emph{greater} slope. Our convention is that $\deg0=0$, so $0$-dimensional classes $\alpha$ have $\widetilde p_\alpha=0\succ\widetilde p_\beta$ whenever $\beta$ is not of dimension 0.)
Then $E\in\mathrm{Coh}\;(X)$ is called \emph{Gieseker} (respectively \emph{tilt}) (semi)stable if for all non-trivial exact sequences $0\to A\to E\to B\to0$ in Coh$\;(X)$ we have
\beq{tidef}
p\_{\;[A]}\ \,(\preceq)\,\ p\_{\;[B]} \quad\(\text{respectively }\ \widetilde p\_{\;[A]}\ \,(\preceq)\,\ \widetilde p\_{\;[B]}\).
\eeq
Here $(\preceq)$ means $\prec$ for stability and $\preceq$ for semistability. In particular (i) ensures that Gieseker and tilt semistable sheaves are pure. 

Fix a class $\alpha\in K(X)$ with $\rk(\alpha)>0$, or $\rk(\alpha)=0$ and $\ch_1(\alpha).H^2>0$. Then an argument from \cite[Proposition 14.2]{Br} shows that an object $E\in\Ab$ of class $\alpha$ is $\nubw$-(semi)stable in the large volume chamber (i.e. for $w\gg0$) if and only if $E\in\mathrm{Coh}\;(X)$ is a tilt (semi)stable sheaf. So the invariants counting tilt semistable sheaves are the $\J_{b,\infty}(\alpha)$. Joyce-Song define an invariant $\J(\alpha)\in\Q$ \cite[Definition 5.15]{JS} counting Gieseker semistable sheaves of class $\alpha$; when there are no strictly semistables this reproduces the integer invariant defined in \cite{Th}.

So we can work in Coh\;$(X)$ to pass from the invariants $\J_{b,\infty}(\alpha)$ counting tilt semistable sheaves to the invariants $\J(\alpha)$ counting Gieseker semistable sheaves. Since tilt stability dominates Gieseker stablity in the sense of \cite[Definition 3.12]{JS} we get a wall crossing formula from \cite[Theorems 3.14 and 5.14]{JS}. It takes the form
\beq{Gti}
\J(\alpha)\=\J_{b,\infty}(\alpha)+\mathop{\sum_{m\ge2,\ \alpha_1,\dots,\;\alpha_m\,\in\,C(\cA),}}_{\sum_{i=1}^m\alpha_i\,=\,\alpha,\ \widetilde p_{\alpha_i}(t)\,=\,\widetilde p_{\alpha}(t)\ \forall i}C(\alpha_1,\dots,\alpha_m)\prod_{i=1}^m\J_{b,\infty}(\alpha_i);
\eeq
see \cite[Equation 78]{FT3}. The universal coefficients $C(\alpha_1,\dots,\alpha_m)$ depend only on the Mukai pairings $\chi(\alpha_i,\alpha_j)$ and the $\prec$-ordering of the reduced Hilbert polynomials $\{p\_{\alpha_i}(t)\}_{i=1}^m$.

The definition \eqref{tidef} of $\widetilde p(t)$ ensures that if $\rk(\alpha)=0,\,\ch_1(\alpha).H^2>0$ then the same is true of the $\alpha_i$ that appear in \eqref{Gti}.
Therefore we can substitute $\alpha=\mathsf v$ into \eqref{Gti} and replace all the terms $\J_{b,\infty}(\alpha),\,\J_{b,\infty}(\alpha_i)$ on the right hand side by their expressions \eqref{nearly} in terms of rank 1 DT invariants $\J(\beta_i)$. This proves Theorem \ref{1}.

%

\section{Dimension one sheaves}\label{dim1}
We again take $K_X=\cO_X$ and $H^1(\cO_X)=0$. In this Section we prove that our technique of using Joyce-Song stable pairs and wall crossing works for dimension 1 classes as well, proving Theorem \ref{2} in this case.

This is just for completeness, as Toda gives a much better, completely explicit formula for the DT invariants $\J(0,0,\beta,m)$ in terms of stable pair invariants in  \cite[Lemma 3.15]{To4}.\medskip

Fix $n \in \mathbb{Z}$ and work in the full subcategory 
 \begin{equation*}
 \cA\ :=\ \Big\langle \cO_X(-n),\ \Coh_{\leq 1}(X)[-1] \Big\rangle_{\text{ext}} \ \subset\ \cD(X),
 \end{equation*}
which is proved in \cite[Lemma 3.5]{TodaDTPT} to be an abelian category. Letting $K(\cA)$ denote its numerical Grothendieck group, define the map 
 \begin{align*}
 \cl \colon K(\cA)&\ \To\ \mathbb{Z}^{\oplus 3},\\
\big[\cO_X(-n)^{\oplus r}\big]+\big[F[-1]\big]&\ \Mapsto\ \(r,\,\ch_2(F).H,\,\ch_3(F)\)\=(r,c,s),
 \end{align*}
for $F\in\Coh_{\le1}(X)$. Equivalently, for $E\in\cA\subset\cD(X)$, 
$$
\cl(E)\=\Big(\!\ch_0(E),\ -\ch_2(E).H + \tfrac12n^2H^3\ch_0(E),\ -\ch_3(E) - \tfrac16n^3H^3\ch_0(E)\Big). 
$$ 
Toda \cite{TodaDTPT} constructs a one dimensional family of weak stability conditions $\nu\_\theta$ on $\cA$, parameterised by $\theta \in \mathbb{R}$. The slope function $\nu\_\theta\colon K(\cA)\to(-\infty,+\infty]$ factors through $\Z^3$ as 
\begin{equation*}
\nu\_{\theta}(F)\=\nu\_{\theta}(r, c, s)\ :=\ \left\{\!\!\!\begin{array}{cl}  \theta & \text{if } r \neq 0,\\ \frac{s}{c} & \text{if } r =0,\ c \neq 0,  \\
+\infty & \text{if } r =c= 0. 
\end{array}\right.
\end{equation*}   	
As usual we say $E \in \cA$ is $\nu\_\theta$-(semi)stable if and only if for any nontrivial short exact sequence $A\into E\onto B$ in $\cA$, we have 
\begin{equation*}
\nu\_\theta(A) \ (\leq) \ \nu\_\theta(B). 
\end{equation*}
Here $(\le)$ means $\le$ for semistability and $<$ for stability. By \cite[Section 5.3]{To2}, $\big\{\nu\_{\theta}\big\}_{\theta \in \mathbb{R}}$ is a continuous family of weak stability conditions satisfying the support property.

The $\cA$-subobjects of $F[-1]\in\Coh_{\le1}(X)$ are the subsheaves of $F$ (shifted by $[-1]$), and the $\nu\_\theta$-slope $s/c$ of a dimension 1 sheaf is, up to scale, the constant term of its reduced Hilbert polynomial. Therefore, for $F\in\Coh_{\le1}(X)\subset\cA$, the following are equivalent,
\begin{itemize}
\item the $\nu\_\theta$-(semi)stability of $F[-1]$ for one fixed $\theta$,
\item the $\nu\_\theta$-(semi)stability of $F[-1]$ for all $\theta\in\R$,
\item the Gieseker (semi)stability of $F$.
\end{itemize}
We now fix $c\in\Z_{>0},\ s\in\Z$ and consider rank 1 objects $E\in\cA$ with class
$$
\cl(E)\=(1,c,s).
$$
Since we have a one parameter family of stability conditions $\{\nu\_\theta\}\_{\theta\in\R}$, the Joyce-Song wall for the class $(1,c,s)$ is now a single point --- the stability condition $\nu\_{\js}=\nu\_{\theta_{\js}}$ for which $[E]$ and $[E]-[\cO_X(-n)]$ have the same slope. It is
$$
\theta\_{\js}\=\frac sc \quad\text{so that}\quad \nu\_{\js}(1,c,s)\=\nu\_{\js}(0,c,s)\=\nu\_{\js}(1,0,0)\=\theta\_{\js}.
$$

\begin{Lem}\label{propp} There are only finitely many walls $\theta\_{\js}=\theta_0<\theta_1<\dots<\theta_k<+\infty$ to the right of $\theta\_{\js}$ on which objects $E$ of class $(1,c,s)$ can be strictly $\nu\_\theta$-semistable. For $\theta$ in any chamber $(\theta_i,\theta_{i+1})$ the semistable objects do not change and are all stable. 
\end{Lem}

\begin{proof}
Note that the classes $(r,c,s)\in\Z^3$ of objects of $\cA$ have $r,c\ge0$, and if $c=0$ then $s\ge0$, because the same is true in $\Coh_{\le1}(X)$. Note also that if an object $E$ of class $(1,c,s)$ is strictly $\nu\_\theta$-semistable then it has a destabilising $\cA$-short exact sequence $E'\into E\onto E''$ with the ranks of $E',\,E''$ being 0 and 1 in some order. We call them $E_0,\,E_1$, where $\rk E_i=i$, and set
$$
\cl(E_0)\=(0,\,c_0,\,s_0) \quad\text{so that}\quad \cl(E_1)\=(1,\,c-c_0,\,s-s_0).
$$
Since $\nu\_\theta(E_0)=\nu\_\theta(E)=\theta<+\infty$ we have $c_0>0$.

Since $\theta>\theta\_{\js}$ is to the right of the Joyce-Song wall,
\begin{equation}\label{one}
 \nu\_\theta(E_0) \= \frac{s_0}{c_0}\= \theta\ >\ \frac{s}{c}\,. 
 \end{equation}
But $ c-c_0 \ge 0$, and if  $c-c_0 = 0$ then $s-s_0\ge0$. The latter contradicts  \eqref{one}, so
 \begin{equation}\label{two}
 0\ <\ c_0\ <\ c.
 \end{equation}
Thus $c_0$ can take only finitely many values, while $s_0\in\Z$ is bounded below by \eqref{one}. If we can find an upper bound for $s_0$ this will prove the finiteness of the possible destabilising classes $\cl(E_0)=(0, c_0, s_0)\in\Z^3$ and walls $\theta=\frac{s_0}{c_0}$ to the right of $\theta\_{\js}$.\medskip

Since $E_1\in\cA$ it is an iterated extension of objects in $\langle\cO_X(-n)\rangle$ and $\Coh_{\le1}(X)[-1]$. Since it has rank 1 we use precisely one term $\cO_X(-n)$ from the former, so collecting the other terms presents $E_1$ as a (non-unique) $\cA$-extension 
$$
E_1'\Into E_1\Onto F[-1] \quad\text{where}\quad F'[-1]\Into E_1'\Onto \cO_X(-n)
$$
for some $F,F'\in\Coh_{\le1}(X)$. That is, $E_1$ is an $\cA$-extension
\beq{normform}
\big[\cO_X(-n)\rt sF'\big]\Into E_1\Onto F[-1].
\eeq
Taking the long exact sequence of cohomology sheaves $\cH^i$ shows that $\cH^0(E_1)$ is the kernel of $\cO_X(-n)\to F'$ and $\cH^1(E_1)$ is an extension of $F$ and the cokernel of $\cO_X(-n)\to F'$. Thus $\cH^0(E_1)$ and $\cH^1(E_1)[-1]$ lie in $\cA$ and the following is an $\cA$-short exact sequence,
$$
\cH^0(E_1)\Into E_1\Onto\cH^1(E_1)[-1].
$$
In other words \emph{we may assume $s$ is onto in \eqref{normform}.} Then writing $\cl(F)=(0,c',s')$, we have
\beq{cGbound}
0\ \le\ c'\=c-c_0-\ch_2(F').H\ \le\ c-c_0.
\eeq
If $c'>0$ then the $\nu\_\theta$-semistability of $E_1$ gives   
 $$
 \frac{s'}{c'}\ \geq\ \nu\_\theta(E_1) \= \theta\ >\ \frac{s}{c}\,,
 $$
 while if $c' = 0$ then $s'\ge0$. So in either case we find $s' \geq \max\(0,\tfrac sc\)$ and therefore
\begin{equation}\label{four}
s-s_0 \= s' +\ch_3(F')\ \ge\ \max\(0,\tfrac sc\)+\ch_3(F').
\end{equation}
Quotients $F'$ of a fixed sheaf $\cO_X(-n)$ with bounded $\ch_2(F').H$ \eqref{cGbound} have $\ch_3(F')$ bounded below; see \cite[Lemma 1.7.9]{HL} for instance. Thus \eqref{four} bounds $s_0$ above, as required. 
\end{proof}

We note for later that these walls $\theta_i$ are also to the right of the Joyce-Song wall $\theta\_{\js}(E_1)=\frac{s-s_0}{c-c_0}$ of the rank 1 semistable factor $E_1$. This follows from \eqref{one} and the see-saw inequality,
\beq{seesaw}
\frac{s_0}{c_0}\=\theta\ >\ \frac sc\ >\ \frac{s-s_0}{c-c_0}\,.
\eeq

We now pick $n\gg0$ so that $H^1(F(n))=0$ for any Gieseker semistable sheaf
\beq{n>>0}
F\,\in\,\Coh_{\le1}(X)\ \text{ with }\ 0\,\le\,\ch_2(F).H\,\leq\,c \ \text{ and }\  \frac{\ch_3(F)}{\ch_2(F).H}\,\ge\,\theta\_{\js} -1.
\eeq
\begin{Thm}\label{thmm} Consider $\nu\_\theta$-semistable objects $E$ of class $(1,c,s)$.
\begin{itemize}
\item For $\theta\in(\theta\_{\js}-1,\theta\_{\js})$ there are none,
\item For $\theta\in[\theta\_{\js},\infty)$ they are complexes $E=[\cO_X(-n)\to F]$ for some $F\in\Coh_{\le1}(X)$, 
\item For $\theta\gg0$ the complex $E(n)$ is a stable pair, and all stable pairs arise in this way.
\end{itemize}
That is, $F$ is pure and $\dim\cH^1(E)=0$ when $\theta\gg0$. 
\end{Thm}

\begin{proof}
 	Take $\theta>\theta\_{\js}-1$ and write any $\nu\_\theta$-semistable object $E$ of class $(1,c,s)$ in the form
\beq{norm3}
 	\big[\cO_X(-n)\rt{s_0}F'\big]\Into E\Onto F''[-1]
\eeq
as in \eqref{normform}, with $F',F''\in\Coh_{\le1}(X)$.
 Let $F'' \onto F^{\min}$ be the the final Gieseker semistable quotient of its Harder-Narasimhan filtration --- i.e. the quotient sheaf of $F''$ of minimum $\ch_3\!/\!\ch_2\!.\;H$ slope. Thus $F^{\min}[-1]$ is the $\cA$-quotient of $F''[-1]$ of minimal $\nu\_\theta$-slope,
and is $\nu\_\theta$-semistable for any $\theta \in \mathbb{R}$.
 
The semistability of $E$ and its $\cA$-quotient $E\onto F^{\min}[-1]$ give the slope inequality
 \begin{equation*}
 \theta\_{\js}-1\ <\ \theta\ \le\ \frac{\ch_3(F^{\min})}{\ch_2(F^{\min}).H}\ \le\ \frac{\ch_3(F^i)}{\ch_2(F^i).H}\,,
 \end{equation*}
 for all Harder-Narasimhan factors $F^i$ of $F''$. Therefore \eqref{n>>0} gives $H^1(F^i(n))=0$ for all $i$, so $H^1(F''(n))=0=\Ext^2(F'', \cO_X(-n))$ by Serre duality. Thus $b\circ a=0$ in the following diagram of exact triangles, enabling us to pick a (not necessarily unique) lift $c$ of $b$ and take (co)cones to complete it to
\begin{equation*}
 \xymatrix@=15pt{
F''[-2]\ar[r]\ar@{=}[d] &F'[-1]\ar[r]\ar[d]+<0ex,1.8ex>&F[-1]\ar[d]\\
F''[-2]\ar[r]^-a &\big[\cO_X(-n) \xrightarrow{s_0} F'\big]\ar[r]\ar[d]_b&E\ar[d]^c\\
&\cO_X(-n)\ar@{=}[r]&\cO_X(-n).
}
\end{equation*}
Since $F$ is an extension of $F'$ and $F''$ it lies in $\Coh_{\le1}(X)$, so the right hand column now shows any $\nu\_\theta$-semistable object $E$ can be written $\cO_X(-n)\rt sF$ when $\theta>\theta\_{\js}-1$. But for $\theta\in(\theta\_{\js}-1,\theta\_{\js})$ the $\cA$-injection $F[-1]\into\big[\cO_X(-n)\to F\big]= E$ destabilises $E$ because
$$
\nu\_{\theta}\(F[-1]\)\=\frac sc\ >\ \theta\=\nu\_\theta(E).
$$

Finally take $\theta>\theta_k$ beyond the final wall for $(1,c,s)$. The $\cA$-surjection $E\onto\cH^{-1}(E)= \cok s\;[-1]$ and the $\nu\_\theta$-semistability of $E$ gives
$$
\theta\=\nu\_\theta(E)\ \le\ \frac{\ch_3(\cok s)}{\ch_2(\cok s).H}\,.
$$
The lack of further walls means this holds for all $\theta>\theta_k$, so we conclude that $\ch_2(\cok s).H$ vanishes, so $\cok s$ has dimension 0.
Furthermore any dimension 0 subsheaf $G\into F$ would define a destabilising $\cA$-injection $G[-1] \hookrightarrow E$, so $F$ is also pure and $E(n)$ is a stable pair.
\medskip

Conversely we claim that twisting any stable pair by $\cO_X(-n)$ gives a $\nu\_{\theta}$-semistable object for $\theta>\theta_k$. So fix $E=\big[\cO_X(-n)\rt sF\big]$ of class $(1,c,s)$ with $F$ pure and $\dim\cok s=0$, and consider a $\nu\_\theta$-destabilising $\cA$-exact sequence
\beq{AB}
A\Into E\Onto B.
\eeq

If $\rk(A)=1$ then $\rk(B)=0$ so $B=\cH^1(B)[-1]\in\Coh_{\le1}(X)[-1]$. The long exact sequence of cohomology sheaves of \eqref{AB} shows $\cH^1(B)$ is a quotient of $\cok s$ and hence 0-dimensional. Thus $\nu\_\theta(B)=+\infty>\theta=\nu\_\theta(E)$ is not destabilising.

If $\rk(B)=1$ then $A\in\Coh_{\le1}(X)[-1]$ and its slope $\nu\_\theta(A)$ is constant in $\theta$. But there are no walls for $E$ beyond $\theta_k$, so \eqref{AB} destabilises for all $0\ll\theta=\nu\_\theta(E)\le\nu\_\theta(A)$. It follows that $\nu\_\theta(A)=+\infty$ and $A=Q[-1]$ for some sheaf $Q$ of dimension 0. Hence $\Hom(Q[-1],\cO_X(-n))=0$ and 
$$
\xymatrix@R=14pt{&Q[-1] \ar@{-->}[dl]_\exists\ar[dr]^0 \\
F[-1] \ar[r]& E \ar@{<-_)}[u]-<0ex,2ex>\ar[r]& \cO_X(-n)}
$$
defines a nonzero map $Q[-1]\to F[-1]$, contradicting the purity of $F$.
\end{proof}

\subsection*{Wall crossing}
By now the wall crossing argument, to express the one dimensional generalised DT invariants $\J(0,0,\beta,m)$ in terms of universal formulae of stable pair invariants, is familiar. We give a brief summary.

We choose a Chern character $(0,0,\beta,m)$ and work with complexes $E$ of
$$
\ch(E)\=\ch(\cO_X(-n))-(0,0,\beta,m)\ \text{ and so }\ \cl(E)\=(1,c,s)\ :=\ (1,\,\beta.H,\,m).
$$
We wall cross from the chamber $\theta\in(\theta_{\js}-1,\theta_{\js})$, where there are no semistable objects, to the large volume chamber $\theta\gg0$ where the semistable objects $E$ are all stable pairs, with counting invariant $P\_{\beta,\,m+n\beta.H}(X)$. As observed in the proof of Lemma \ref{propp}, on any wall we have destabilising factors $E_0,\,E_1$ of ranks $0,\,1$ respectively, both with smaller values of $c$ \eqref{two} and with $E_0$ a Gieseker semistable sheaf.

When $c=c_{\min}>0$ takes the minimal possible value this means the only possibility is that $\cl(E_0)=(0,c_{\min},s)$, so $\cl(E_1)=(1,0,0)$ and the wall is $\theta_{\js}$. It follows that $E_0$ is Gieseker stable and the wall crossing formula is the simplest one \eqref{JJJ}. Thus we find
\beq{easyJS}
P\_{\beta,\,m+n\beta.H}(X)\=(-1)^{m+n\beta.H-1}(m+n\beta.H)\,\J(0,0,\beta,m),
\eeq
expressing $\J(0,0,\beta,m)$ in terms of stable pair invariants.

Now take general $c>0$. On each wall the semistable factors $E_0,\,E_1$ either have the classes $(1,0,0),\,(0,c,s)$ again, or have \emph{strictly smaller $c$}. In the first case we are on $\theta_{\js}$ and again pick up the term on the right hand side of \eqref{easyJS}, plus higher order terms counting semistable factors of $E_0$; since these have smaller $c$ we may assume by induction they are written in terms of stable pair invariants.

In the second case both $E_0,\,E_1$ have strictly smaller $c>0$, and $\theta$ is to the right of the Joyce-Song wall of $E_1$ by \eqref{seesaw}. By induction we may assume both $\J(0,0,\beta',m')$ and $\J_\theta(1,\beta',m')$ can be written in terms of stable pair invariants for all $\beta'.H<c$ and $\theta$ to the right of the Joyce-Song wall of $(1,\beta',m')$, since the base case $\beta'.H=c_{\min}$ was shown above. After finitely many walls, we reach the large volume chamber where we get the stable pair invariant $P\_{\beta,\,m+n\beta.H}(X)$ equated with the right hand side of \eqref{easyJS} plus a function of stable pair invariants with smaller $c$. Rearranging --- and rewriting the stable pair invariants in terms of counts of ideal sheaves using the DT/PT wall crossing proved in \cite{BrDTPT, TodaDTPT}
 --- proves Theorem \ref{2} for dimension 1 classes.\medskip

Finally we note that the working of \cite[Section 6.3]{JS} shows that the same technique of using Joyce-Song pairs also works for dimension 0 sheaves.

\appendix
\section{Stable pairs}\label{appendix}
We explain here some results of Toda. Though they are all proved in \cite[Section 3]{TodaBG} they are not stated in the generality we need, so we give a brief account (with slightly different proofs) for completeness.

\begin{Def}
	A pair $(F,s)$ consisting of a 1-dimensional sheaf $F$ and a section $s\in H^0(F)$ is called a \emph{stable pair} \cite{PT} if 
	\begin{itemize}
		\item $F$ is a \emph{pure} 1-dimensional sheaf, i.e. it has no 0-dimensional subsheaves, and 
		\item $s \colon \cO_X \rightarrow F$ has zero-dimensional cokernel. 
	\end{itemize}
\end{Def}

We often abuse notation and call the 2-term complex $\cO_X\rt sF$ (with $\cO_X$ in degree 0) a stable pair. We also use the notation $E^\vee:=R\hom(E,\cO_X)$ for derived dual.

	\begin{Lem}\label{lem1}
	Suppose $E \in \Ab$ is a $\nu\_{b,w}$-semistable object of rank $-1$ with $\nubw(E)<+\infty$ and $w \gg 0$. Set $L:=(\det E)^{-1}$. Then
$$
E^\vee\otimes L[1]\,\text{ is a stable pair}.
$$	
	\end{Lem}
	
	\begin{proof}
    By \cite[Lemma 2.7(c)(iii)]{BMS} we have $\cH^{0}(E) \in \Coh_{\leq 1}(X)$ while $\cH^{-1}(E)$ is a $\mu\_H$-semistable sheaf. In particular $\cH^{-1}(E)$ is torsion free, and rank 1 by Lemma \ref{lem.destabilising objects}. Thus it is $L\otimes I_Z$ for some subscheme $Z\subset X$ of dimension $\le1$ and line bundle $L$ --- the double dual of $\cH^{-1}(E)$, which also equals $\det\!\(E[1]\)$ since $Z$ and $\cH^0(E)$ have dimension $\le$ 1.
    
    Since $\mu\_H(L)=\mu\_H\(\cH^{-1}(E)\)$ the shifts by $[1]$ of both sheaves are in $\Ab$. This gives two $\Ab$ short exact sequences
	\begin{equation}\label{sequences}
     	L\otimes\cO_Z \Into\cH^{-1}(E)[1]\Onto L[1] \qquad\text{and}\qquad \cH^{-1}(E)[1]\Into E\Onto\cH^0(E)
	\end{equation}
and so the $\Ab$-injection $L\otimes\cO_Z\into E$	. Since $\nu\_{b,w}(L\otimes\cO_Z) = +\infty$ this contradicts the $\nubw$-semistability of $E$ unless $Z=\emptyset$. Thus $\cH^{-1}(E)=L$.\medskip
	
The derived dual of the second exact sequence of \eqref{sequences} is the exact triangle
$$
\cH^0(E)^\vee\To E^\vee\To L^{-1}[-1].
$$
Setting $F:=\cH^0(E)^\vee\otimes L[2]$ and $I\udot:=E^\vee\otimes L[1]$ this gives
\beq{tria}
I\udot\To\cO_X\To F.
\eeq
Since $\dim\cH^0(E)\le1$, we know $\cH^0(F)=\ext^2(\cH^0(E)\otimes L^{-1}, \cO_X)$ is a pure 1-dimensional sheaf, $\cH^1(F)=\ext^3(\cH^0(E)\otimes L^{-1}, \cO_X)$ is a 0-dimensional sheaf, and the other $\cH^i(F)$ vanish. The long exact sequence of cohomology sheaves of \eqref{tria} then shows that
\beq{malv}
\cH^{\ge3}(I\udot)\,=\,0\ \text{ and }\ \cH^2(I\udot)\,=\,\cH^1(F)\ \text{ is 0-dimensional.}
\eeq

Now we use the $\nubw$-semistability of $E\in\Ab$ and $\nubw(E)<+\infty$ to deduce
\beq{homfe}
\Hom(G,E)\=0\=\Hom(E^\vee,G^\vee)\ \text{ for any }\ G\,\in\,\mathrm{Coh}_{\le1}(X).
\eeq
Setting $G=\cO_x$ for any point $x\in X$ gives
$$
0\=\Hom\!\(I\udot\otimes L^{-1}[-1],\cO_x[-3]\)\ \cong\ \Hom(I\udot[2],\cO_x)\=\Hom\!\(\cH^2(I\udot),\cO_x\),
$$
where the last equality follows from $\cH^{\ge3}(I\udot)=0$ \eqref{malv}. Combined with the second part of \eqref{malv} this gives $\cH^2(I\udot)=0$, so in fact
\beq{ge2}
\cH^{\ge2}(I\udot)\,=\,0\,\text{ and }\,F\,\text{ is a pure 1-dimensional sheaf}.
\eeq
Now let $G$ be any pure 1-dimensional sheaf. Then $G^\vee\otimes L[2]$ is also a pure 1-dimensional sheaf, which we may also subsitute into \eqref{homfe} in place of $G$. This gives
\beq{end}
0\=\Hom\!\(E^\vee,(G^\vee\otimes L[2])^\vee\)\=\Hom(I\udot[1],G)\ \stackrel{\eqref{ge2}}=\ \Hom\!\(\cH^1(I\udot),G\).
\eeq
But by \eqref{tria} $\cH^1(I\udot)$ is a quotient of a pure 1-dimensional sheaf $F$, so by \eqref{end} it must be 0-dimensional. Thus by \eqref{tria}, $I\udot$ is the cocone of a map $\cO_X\to F$ with 0-dimensional cokernel, i.e. it is a stable pair.
  \end{proof}
  
For the converse let $(F,s)$ be a stable pair and $I\udot=\{\cO_X\rt sF\}$ the associated complex.

  \begin{Lem}\label{lem2}
For any line bundle $L$ take $b> L.H^2/H^3$ and $w \gg 0$. Then 
$$
(I\udot)^\vee\otimes L[1]\,\in\,\Ab\,\text{ and is }\,\nu\_{b,w}\text{-stable}.
$$ 
\end{Lem}  
    \begin{proof} The dual\,$\otimes L[1]$ of the exact triangle $I\udot\to\cO_X\to F$ gives 
\begin{equation}\label{L[1]}
L [1]\To L \otimes (I\udot)^\vee[1] \To L \otimes F^\vee[2].
\end{equation}
Since $F$ is a pure 1-dimensional sheaf, so is $L\otimes F^\vee[2]$. So the exact triangle shows the rank $-1$ complex $E:=L \otimes (I\udot)^\vee[1]$ lies in $\Ab$ \eqref{Abdef} when $b>\mu\_H(L)=L.H^2/H^3$.

By the wall and chamber structure of Proposition \ref{prop. locally finite set of walls} and the primitivity of $\ch(E)$ it is sufficient to prove that $E$ is $\nubw$-semistable for a fixed $b > \max\(0, \mu\_H(L)\)$ and $w\gg0$. Since $\rk(E)=-1$ the slope $\nubw(E)$ is linear and increasing in $w$. So take $w$ satisfying
\begin{enumerate}
\item[(i)] $\nubw(E)\,>\,0$,
\item[(ii)] $\nubw(E)\,>\,\ch_2(E).H+\tfrac1{2H^3}\max\!\((bH^3)^2,\,(L.H^2)^2\)$, and
\item[(iii)] $w\,>\,\tfrac12\(\!\ch_1^{bH}(E).H^2\)^2+b\ch_1^{bH}(E).H^2+\tfrac12b^2$,
\end{enumerate}
and suppose for a contradiction that $E$ is strictly $\nubw$-unstable. Taking the first term of its $\nubw$-Harder-Narasimhan filtration gives an $\Ab$-exact sequence $E' \hookrightarrow E \twoheadrightarrow E''$ with $\nubw$-semistable $E'$ and
	\begin{equation}\label{order-slope}
	\nubw(E')\ >\ \nubw(E)\ \ge\ \nubw(E'').
	\end{equation}
Thus $\cH^0(E'')$ is a quotient of the 1-dimensional sheaf $\cH^0(E)$, so it is supported in dimension $\leq 1$. If $\cH^{-1}(E'')$ also has rank 0 then it vanishes by \eqref{Abdef}, so $\rk(E'')=0$ and $\nubw(E'')=+\infty$, contradicting \eqref{order-slope}. We conclude $\rk(E'')<0$ and $\rk(E')\ge0$.

Suppose first that $\rk(E')>0$. Then $\rk(\cH^0(E'))>0$ so \eqref{65} gives $\ch_1^{bH}(E').H^2>0$. By (i) and \eqref{order-slope} we also have $\nubw(E')>0$. So their product is also positive,
\beqa
0 &<& \ch_2(E').H-w\ch_0(E')H^3 \\
&\stackrel{\eqref{BOG}}\le& \frac{(\ch_1(E').H^2)^2}{2\ch_0(E')H^3}-w\ch_0(E')H^3 \\
&=& \frac{(\ch_1^{bH}(E').H^2)^2}{2\ch_0(E')H^3}+b\ch_1^{bH}(E').H^2+\tfrac12b^2\ch_0(E')H^3-w\ch_0(E')H^3 \\
&\le& \tfrac12(\ch_1^{bH}(E).H^2)^2+b\ch_1^{bH}(E).H^2+\tfrac12b^2\ch_0(E')H^3-w\ch_0(E')H^3.
\eeqa
Therefore $w<\frac12(\ch_1^{bH}(E).H^2)^2+b\ch_1^{bH}(E).H^2+\tfrac12b^2$, which contradicts (iii).

So $\rk(E')=0$ and $\rk(E'')=-1$. We claim that $E'$ is a sheaf. If not then since $\cH^{-1}(E')$ and $\cH^{-1}(E'')$ are torsion free \eqref{Abdef} and $\rk\!\(\cH^{-1}(E)\)=1$, the exact sequence
$$
	0\To\cH^{-1}(E')\rt{\alpha}\cH^{-1}(E)\rt{\beta}\cH^{-1}(E'')\To\cH^0(E')\To\cH^0(E)\To\cH^0(E'')\To0
$$
shows the map $\alpha$ is an isomorphism and $\beta$ is 0. Since $\cH^0(E)$ has dimension $\le1$ the rest of this sequence then shows that $\mu\_H\(\cH^{-1}(E'')\)=\mu\_H\(\cH^0(E')\)$, which is impossible by the definition \eqref{Abdef} of $\Ab$.

So $E'$ is a rank 0 sheaf. The inclusion $\cH^{-1}(E) \hookrightarrow \cH^{-1}(E'')$ implies the rank 1 torsion free sheaf $\cH^{-1}(E'')$ is $L''\otimes I_{Z''}$ for some $Z''$ of $\dim\le1$ and line bundle $L''$ with 
	\begin{equation*}
	L.H^2\ \leq\ L''.H^2\ \stackrel{\eqref{Abdef}}\leq\ bH^3. 
	\end{equation*} 
The Hodge index theorem $\((L'')^2.H\)(H^3)\le(L''.H^2)^2$ and $\ch_2\!\(\cH^0(E'')\).H\ge0$ then give
	\begin{equation*}
	\ch_2(E'').H\ \ge\ -\ch_2(L''\otimes I_{Z''}).H\ \ge\ -\tfrac12(L'')^2.H\ \geq\ - \tfrac{1}{2H^3} \max\((bH^3)^2 \,,\,(L.H^2)^2\).
	\end{equation*}
If $\ch_1(E').H^2>0$ then we get the following contradiction to (ii) and \eqref{order-slope},
	$$
\nubw(E')\=\frac{\ch_2(E').H}{\ch_1(E').H^2}\ \le\ \ch_2(E').H\ \le\ \ch_2(E).H+\tfrac1{2H^3}\max\((bH^3)^2 \,,\,(L.H^2)^2\).
	$$
	So in fact the rank 0 sheaf $E'$ is supported in dimension $\le1$. \medskip
	
	From \eqref{L[1]} we get the following diagram in $\Ab$,
$$
\xymatrix@R=18pt{
L[1]\ \ar@{<-_)}[d]+<0pt,9pt>\ar@{^(->}[r]<-.2ex>& L \otimes (I\udot)^\vee[1] \ar@{<-_)}[d]+<0pt,10pt><.5ex>\ar@{->>}[r]&  L \otimes F^\vee[2] \ar@{<-_)}[d]+<0pt,9pt> \\
\ker\gamma\ \ar@{^(->}[r]<-.2ex>& E' \ar[ur]_(.45){\gamma}\ar@{->>}[r]& \im\gamma.}
$$
Since $\Ab$-subobjects of sheaves are sheaves, both $\ker\gamma$ and $\im\gamma$ are sheaves. Thus the exact sequence of cohomology sheaves of the lower row shows $\ker\gamma$ is a \emph{subsheaf} of $E'$, so it has dimension $\le1$. The $\Ab$-cokernel of the left hand injection has torsion free $\cH^{-1}$ sitting in an exact sequence of sheaves $0\to L\to\cH^{-1}\to\ker\gamma\to0$, which forces $\ker\gamma=0$.

Thus we have the $\Ab$-injection $\gamma\colon E'\into L \otimes F^\vee[2]$. Hence $\cH^{-1}(\cok\gamma)$ is both torsion free and a subsheaf of $E'$; it is therefore zero and $E'\into L \otimes F^\vee[2]$ is also an injection of sheaves. We deduce $E'$ is a pure 1-dimensional sheaf.

The dual\,$\otimes\;L[1]$ of $E'\into E\onto E''$ gives
\beq{exact-dual}
(E''\otimes L^{-1})^\vee[1]\To I\udot\To(E')^\vee\otimes L[1].
\eeq
Since $E''\in\Ab$ we have $E'' \otimes L^{-1}\in\cA_{\;b'}$, where $b' = b - L.H^2/H^3$. And $\nu^{\max}_{b,w}(E'')<+\infty$
by the definition of $E''$ via the $\nubw$-Harder-Narasimhan filtration, so $\nu^{\max}_{b', w}(E'' \otimes L^{-1})<+\infty$.
Therefore \cite[Proposition 5.1.3(b)]{BMT} applies to show $\cH^{\ge2}\((E'' \otimes L^{-1})^\vee[1]\) = 0$. Then the exact sequence of cohomology sheaves of \eqref{exact-dual} gives a surjection
$$
\cok(s) \Onto \cH^{2}\((E')^\vee\otimes L\)
$$
from a 0-dimensional sheaf to a nonzero 1-dimensional sheaf, a contradiction.
\end{proof}

The bijection of sets defined by Lemmas \ref{lem1} and \ref{lem2} can be upgraded to an isomorphism of moduli stacks. There is a fine moduli space $P$ of stable pairs of fixed Chern character. It admits a universal stable pair on $P\times X$, flat over $P$. Taking its derived dual and twisting by $L[1]$ defines a perfect complex over $P\times X$. By Lemma \ref{lem2} its restriction to any $\{$point$\}\times X$ is a $\nu\_{b,w}$-stable complex for $bH^3>L.H^2$ and $w\gg0$.

These complexes all have the same class $v\in K(X)$. There is an algebraic moduli stack of finite type $\cM_{b,w}(v)$ of $\nubw$-semistable objects in this class \cite[Theorem C.5]{FT3}.

The perfect complex is classified by a map
$P\to\cM_{b,w}(v)$ which factors through $P/\C^*$. (Here the $\C^*$ acts trivially on $P$ but scales the complex.) Twisting further by elements of $\Pic\_0(X)$ defines a map
\beq{stacks}
P/\C^*\times\Pic\_0(X)\To\cM_{b,w}(v)
\eeq
which is a bijection by Lemmas \ref{lem1} and \ref{lem2}. But deformations of stable pairs $(F,s)$ and the complexes $(I\udot)^\vee\otimes L[1]$ coincide to all orders \cite[Theorem 2.7]{PT}, so the bijection induces an isomorphism $P\times\Pic\_0(X)\to M_{b,w}(v)$ on coarse moduli spaces. Since the stabiliser groups are $\C^*$ on both sides, we deduce that \eqref{stacks} is an isomorphism of stacks.

Hence, for $bH^3>L.H^2$, the invariants $\J_{b,\infty}(v)$ equal the stable pair invariants of \cite{PT} multiplied by $\#H^2(X,\Z)_{\mathrm{tors}}$.

%
%
%
%
%
%
%

\bibliographystyle{halphanum}
\bibliography{references}

\bigskip \noindent
{\tt{s.feyzbakhsh@imperial.ac.uk\\ richard.thomas@imperial.ac.uk}}\medskip

\noindent Department of Mathematics\\
\noindent Imperial College\\
\noindent London SW7 2AZ \\
\noindent United Kingdom

\end{document}